\documentclass[12pt,reqno]{amsart}%
\usepackage{amsmath, amsfonts, amssymb, amsthm, amscd, amsbsy}
\usepackage{fancyhdr}
\usepackage[usenames,dvipsnames,svgnames,x11names,hyperref]{xcolor}
\usepackage{geometry}
\usepackage{graphicx}
\usepackage{enumerate}
\usepackage[pagebackref]{hyperref}%
\usepackage{amsmath}%
\usepackage{amsfonts}%
\usepackage{amssymb}
\providecommand{\U}[1]{\protect\rule{.1in}{.1in}}
\hypersetup{backref=true,
pagebackref=true,
hyperindex=true,
colorlinks=true,
breaklinks=true,
urlcolor=NavyBlue,
linkcolor=Fuchsia,
bookmarks=true,
bookmarksopen=false,
filecolor=black,
citecolor=ForestGreen,
linkbordercolor=red
}
\newtheorem{theorem}{Theorem}[section]
\newtheorem{theoremintro}{Theorem}
\theoremstyle{plain}

\newtheorem{corollary}[theorem]{Corollary}

\newtheorem{example}{Example}[section]

\newtheorem{lemma}[theorem]{Lemma}

\newtheorem{proposition}[theorem]{Proposition}
\newtheorem{remark}[theorem]{Remark}

\numberwithin{equation}{section}

\newtheorem{theoremalph}{Theorem}

\newtheorem{lemmaalph}{Lemma}

\allowdisplaybreaks

\def\a{\alpha}
\def\b{\beta}
\def\d{\delta}
\def\e{\epsilon}
\def\la{\lambda}
\def\p{\partial}
\def\g{\gamma}
\def\grad{\nabla}
\def\oo{\infty}

\def\dist{\operatorname{dist}}
\def\R{\mathbb{R}}

\def\mcH{\mathcal{H}}

\def\esssup{\operatorname{ess}\sup}

\newcommand{\gen}[1]{\ensuremath{\langle #1\rangle}}

\def\diam{\operatorname{diam}}

\newcommand{\norm}[1]{\left\lVert#1\right\rVert}

\begin{document}
\title[Hardy Identities on Domains]{$L^p$-Hardy identities and inequalities with respect to the distance and mean distance to the boundary}
\author{Joshua Flynn}
\address{Joshua Flynn: Department of Mathematics\\
University of Connecticut\\
Storrs, CT 06269, USA}
\email{joshua.flynn@uconn.edu}
\author{Nguyen Lam}
\address{Nguyen Lam: School of Science and the Environment\\
Grenfell Campus, Memorial University of Newfoundland\\
Corner Brook, NL A2H5G4, Canada}
\email{nlam@grenfell.mun.ca}
\author{Guozhen Lu }
\address{Guozhen Lu: Department of Mathematics\\
University of Connecticut\\
Storrs, CT 06269, USA}
\email{guozhen.lu@uconn.edu}
\thanks{J. Flynn and G. Lu were partially supported by a collaboration grant  from the Simons Foundation. N. Lam was partially supported by an NSERC Discovery Grant.}
%\date{\today}

\begin{abstract}
  Firstly, this paper establishes useful forms of the remainder term of Hardy-type inequalities on general domains where the weights are functions of the distance to the boundary.
  For weakly mean convex domains we use the resulting identities to establish nonexistence of extremizers for and improve known sharp Hardy inequalities.
  Secondly, we establish geometrically interesting remainders for the Davies-Hardy-Tidblom inequalities for the mean distance function, as well as generalize and improve several Hardy type inequalities   in the spirit of Brezis and Marcus and  spectral estimates of Davies. Lastly, we apply our results to obtain Sobolev inequalities for non-regular Riemannian metrics on geometric exterior domains.
\end{abstract}


\maketitle
\tableofcontents

\section{Introduction}

The objective of this paper is to study applications, generalizations and improvements of the classical Hardy inequalities of the following form on domains: for $u\in{C_{0}^{\oo}(\Omega)}$ and $1<p<\oo$, the classical Hardy inequality on a domain $\Omega \subsetneq \R^{N}$ is given by
\begin{equation}
  \int\limits_{\Omega}|\grad{u(x)}|^{p}dx\geq{c_{p}}\int\limits_{\Omega}\frac{|u(x)|^{p}}{d(x)^{p}}dx,
  \label{eq:hardy-inequality-general-domain}
\end{equation}
where $c_{p}=c(p,N,\Omega)$ is a positive constant independent of $u$ and $d$ is either the distance function or mean distance function to the boundary $\p\Omega$.
To be more precise, let
\begin{align*}
  d_{\Omega}(x)&=\dist(x,\p\Omega)\\
  d_{\Omega,M,p}(x)&=\left( \frac{\sqrt{\pi}\Gamma\left( \frac{N+p}{2} \right)}{\Gamma\left( \frac{p+1}{2} \right)\Gamma\left( \frac{N}{2} \right)}\int\limits_{S^{N-1}}\rho_{\nu}(x)^{-p}d\sigma(\nu)\right)^{-\frac{1}{p}}
\end{align*}
be, respectively, the distance and mean distance functions to the boundary $\p\Omega$, where $1<p<\oo$, $\rho_{\nu}(x)$ is the distance from $x\in\Omega$ to $\p\Omega$ in the direction $\nu\in{S^{N-1}}$ and $d\sigma(\nu)$ is the normalized round measure on the sphere $S^{N-1}$.
Then \eqref{eq:hardy-inequality-general-domain} holds firstly for $d=d_{\Omega}$ when either we impose appropriate regularity assumptions on $\p\Omega$ or we restrict the exponent range to $N<p<\oo$ and secondly for all exponents $1<p<\oo$ and all proper sub-domains when $d=d_{\Omega,M,p}$.

We first describe the essence of this paper in generality.
So, suppose we have two functional quantities $F(u)$ and $G(u)$ which are related by the sharp inequality $c F(u) \leq  G(u)$, where $c$ is the largest positive constant for this inequality to hold for all $u$ in some function space.
In a word, the primary interest of this paper is to find useful forms of the remainder term $\d(u) := G(u) - cF(u)$, which will serve several purposes depending on its form.
In case there is an extremal $u^{\sharp}$ attaining the sharp constant $c$, i.e., $cF(u^{\sharp}) = G(u^{\sharp})$, we have $\d(u^{\sharp}) = 0$ and hence a means to determine the existence and form of extremals $u^{\sharp}$, provided we can solve $\d = 0$.
In case there are no extremals $u^{\sharp}$ for $cF(u) \leq G(u)$, one strives to find meaningful $\d_{1}(u) \geq 0$ so that $\d(u) - \d_{1}(u) \geq 0$, thereby implying the improved inequality $cF(u) + \d_{1}(u) \leq G(u)$.
Lastly, independent of these two ideas, sometimes $\d(u)$ is already expressed as a useful functional quantity related to another $H(u)$ via $c' H(u) \leq \d(u)$.
In this case, the inequalities may be combined to create the new inequality $c'H(u) \leq G(u)$.

To be more specific, we focus on finding useful expressions of the remainders $\d$ and $\d_{M}$ in the following weighted Hardy identities:
\begin{equation}
  \int\limits_{\Omega}V(d_{\Omega}(x))|\grad{u(x)}|^{p}dx-\int\limits_{\Omega}W(d_{\Omega}(x))|u(x)|^{p}dx= \d(u)
  \label{eq:intro-hardy-identity-one}
\end{equation}
and
\begin{equation}
  \int\limits_{\Omega}\tilde{V}_{M,p}(x)|\grad{u(x)}|^{p}dx-\int\limits_{\Omega}W_{M,p}(x)|u(x)|^{p}dx= \d_{M}(u).
  \label{eq:intro-hardy-identity-two}
\end{equation}
%where, for some $\a\in\R$, $(r^{\a-1}V,r^{\a-1}W)$ is a ($p$-Bessel) pair of $C^{1}$-weights on an interval $(0,R)$ such that there exists a positive $C^{1}$-function $\varphi$ satisfying the (generally nonlinear) ODE
Here, $(V,W)$ is a ($p$-Bessel) pair of positive $C^{1}$-weights on an interval $(0,R)$ such that there exists a positive $C^{1}$-function $\varphi$ satisfying the (generally nonlinear) ODE
\begin{equation}\label{BesselPair}
  %(r^{\a-1}V|\varphi'|^{p-2}\varphi')'+r^{\a-1}W|\varphi|^{p-2}\varphi=0
  (V|\varphi'|^{p-2}\varphi')'+W|\varphi|^{p-2}\varphi=0
\end{equation}
on $(0,R)$, and $\tilde{V}_{M,p}$ and $W_{M,p}$ are certain spherical means of $V$ and $W$, respectively, that generalize powers of $d_{\Omega,M,p}$ (see \eqref{eq:spherical-means-of-bessel-pairs} for a precise definition). This notion of $p-$Bessel pair $(V, W)$, initially defined in the work of Duy, Lam and Lu \cite{DLL},  is an extension of the $2-$ Bessel pair first introduced by Ghoussoub and Moradifam \cite{GM, GM2}.
In light of the general discussion above, it is clear that, when $\d$ or $\d_{M}$ are nonnegative or nonnegative and contain useful improving terms, then these identities imply generalizations or improvements of \eqref{eq:hardy-inequality-general-domain}, respectively.
(Naturally, specially chosen pairs $(V,W)$ recover \eqref{eq:hardy-inequality-general-domain}.)
Moreover, our expressions for $\d$ allow us to conclude nonexistence of extremal functions for the Hardy inequalities arising from \eqref{eq:intro-hardy-identity-one} whenever they are sharp.
Lastly, if $\d_{M} \geq0$, then \eqref{eq:intro-hardy-identity-two} identifies a new class of weights for which a Hardy-type inequality holds.

Concerning the class of domains we consider, the main regularity assumption we will impose on the domain boundaries $\p\Omega$ is that they are weakly mean convex, i.e., they have everywhere non-positive mean curvature (our sign convention is that convex boundaries have non-positive principal curvatures).
The family of such domains includes nonconvex domains such as solid tori in $\R^{3}$ whose outer radius is at least twice the inner radius; e.g., see \cite[Example 2.5.2]{BET}.
Other ring-like domains which are weakly mean convex are just as easy to construct.
Our results also apply to weakly mean convex domains bounded by minimal surfaces and which have finite inradius (e.g., either region bounded by the Schwarz P surface).

\medskip

Our Hardy identities have several applications and these are detailed in the discussion below.
To highlight some of the main applications, we collect four main results in the following four theorems.
%To highlight some of the main applications, we collect three main applications in the following three theorems.
%Theorems \ref{thm:main-theorem1} and \ref{thm:main-theorem2} and related results are contained and proved in Section \ref{section:hardy-identities}.
The first theorem (Theorem \ref{thm:main-theorem1}) and related results are contained and proved in Section \ref{section:hardy-identities}.
The second and third theorems (Theorems \ref{thm:main-theorem2} and \ref{thm:main-theorem3}) and related results are contained and proved in Section \ref{section:mean-distance-function-results}.
Lastly, the third theorem (Theorem \ref{thm:main-theorem4}) and related results are contained and proved in Section \ref{section:sobolev-inequality}.

\medskip

Theorem \ref{thm:main-theorem1} pertains to studying the Hardy remainder and resulting existence results for extremals in the corresponding Hardy inequalities on weakly mean convex domains.
In preparation, fix $N \geq 2$, let $1 < p < \oo$, let $\Omega \subsetneq \R^{N}$ be a domain with inradius $0 < \rho \leq \oo$ and let $(V,W)$ be a $p$-Bessel pair on $(0,R)$ where $R=\oo$ or $R > \rho$.
Lastly, define the weighted Sobolev space $W_{0}^{1,p}(\Omega,V,W)$ as the closure of $C_{0}^{\oo}(\Omega)$ with respect to the norm
\[
  \left( \int\limits_{\Omega}V(d_{\Omega})|\grad{u}|^{p}dx + \int\limits_{\Omega}W(d_{\Omega})|u|^{p}dx \right)^{\frac{1}{p}}.
\]

Recall that $\varphi$ is a positive solution to the Bessel pair $(V, W)$ definition equation \eqref{BesselPair}.

\begin{theoremintro}[Theorem \ref{thm:bessel-pair-hardy-identity-general-domain}; Corollary \ref{cor:improvements-corollary}.]\label{thm:main-theorem1}
  Let $N,p,\Omega,\rho,(V,W),R,\varphi$ be as above and let
  \[
    \d(u) = \int_{\Omega} V(d_{\Omega}) |\grad u|^{2}dx - \int_{\Omega} W(d_{\Omega}) |u|^{p} dx
  \]
  be the remainder.
  If $\Omega$ is weakly mean convex and $\varphi$ is increasing, then $\d(u) \geq 0$ for $u \in W_{0}^{1,p}(\Omega,V,W)$ and $\d(u) = 0$ iff $u = 0$; in particular
  \[
    \int_{\Omega} W(d_{\Omega}) |u|^{p} dx < \int_{\Omega} V(d_{\Omega}) |\grad u|^{2}dx
  \]
  for all nonzero $u \in W_{0}^{1,p}(\Omega,V,W)$.
  Moreover, $\d(u)$ may be expressed as a distribution depending only on $u$, $\grad u$, $p$, $d_{\Omega}$, $V$ and $\varphi$.
\end{theoremintro}

% \begin{theoremintro}\label{thm:main-theorem1}
%   Let $1 < p<\oo$, let $N \geq 1$, let $\Omega\subsetneq\R^{N}$ be a weakly mean convex domain with inradius $0<\rho\leq\oo$ and let $0<R\leq\oo$ be such that $\rho < R$ when $\rho<\oo$ and $R=\oo$ when $\rho = \oo$.
%   Suppose $(V,W)$ is a $p$-Bessel pair on $(0,R)$ with positive increasing solution $\varphi$.
%   Then there exists an explicit nonnegative functional $F$ such that, for $u \in W_{0}^{\oo}(\Omega,V,W)$, there holds:
%   \[
%     \int\limits_{\Omega}V(d_{\Omega}(x))|\grad{u(x)}|^{p}dx-\int\limits_{\Omega}W(d_{\Omega}(x))|u(x)|^{p}dx= F(|u|^{p},|\grad u|^{p}, d_{\Omega}, V, \varphi).
%   \]
%    %  \item For a general domain $\Omega$ with finite essential diameter $D_{\oo}(\Omega) = 2R$, there exists a nonnegative functional $F_{M}$ such that, for $u \in C_{0}^{\oo}(\Omega)$, there holds:
%
%    %    \[
%    %      \int\limits_{\Omega}\tilde{V}_{M,p}(x)|\grad{u(x)}|^{p}dx-\int\limits_{\Omega}W_{M,p}(x)|u(x)|^{p}dx= F_{M}(|u|^{p},|\grad u|^{p}, \rho_{\nu}, d_{\Omega,M,p}, \tilde{V}_{M,p}, \varphi).
%    %    \]
% \end{theoremintro}

We note that Theorem \ref{thm:main-theorem1} will be stated in more details, including the exact form of $\d(u)$, and will be proved in Theorem \ref{thm:bessel-pair-hardy-identity-general-domain} and Corollary \ref{cor:improvements-corollary} in Section \ref{section:hardy-identities}.

By choosing specific $p$-Bessel pairs $(V, W)$, Theorem \ref{thm:main-theorem1} implies and improves results in the literature.
%This theorem also addresses the existence of extremizers; see Corollary \ref{cor:improvements-corollary} whose hypotheses are slightly more general.
%Moreover, since the identity in Theorem \ref{thm:main-theorem1} implies many known Hardy inequalities, our Theorem \ref{thm:main-theorem2} establishes the nonexistence of extremizers for many known Hardy inequalities.
To obtain the existence result, we first prove that the existence of a nontrivial solution to $\d(u)=0$ implies $\Omega$ contains a hyperplane.
It is then easy to see that nontrivial solutions to $\d(u)=0$ cannot be integrable near $\p\Omega$; see the proof of Corollary \ref{cor:improvements-corollary}, where slightly more general hypotheses are assumed.

Next, Theorem \ref{thm:main-theorem2} pertains to establishing new Hardy inequalities for weights obtained by taking spherical means of Bessel pairs composed with the distance function to the boundary of a domain.
We also obtain geometrically interesting expressions for the remainder terms.

\begin{theoremintro}[Theorem \ref{thm:mean-bessel-pair-hardy-identities}]\label{thm:main-theorem2}
  Let $N,p,\Omega,\rho,(V,W),R,\varphi$ be above, let $(\tilde{V}_{M,p},W_{M,p})$ be the spherical means of $(V,W)$ given by \eqref{eq:spherical-means-of-bessel-pairs} and let
  \[
    \d_{M}(u) = \int\limits_{\Omega}\tilde{V}_{M,p}(x)|\grad{u(x)}|^{p}dx-\int\limits_{\Omega}W_{M,p}(x)|u(x)|^{p}dx
  \]
  be the remainder.
  If $\varphi$ is increasing, then $\d_{M}(u) \geq 0$ for $u \in C_{0}^{\oo}(\Omega)$; in particular
  \[
    \int\limits_{\Omega}W_{M,p}(x)|u(x)|^{p}dx \leq  \int\limits_{\Omega}\tilde{V}_{M,p}(x)|\grad{u(x)}|^{p}dx.
  \]
  Moreover, $\d_{M}(u)$ may be expressed as a distribution depending only on $u$, $\grad u$, $p$, $d_{\Omega}$, $V$ and $\varphi$.

\end{theoremintro}

The exact form of $\d_{M}(u)$ can be found in Theorem \ref{thm:mean-bessel-pair-hardy-identities}. The proof of Theorem \ref{thm:mean-bessel-pair-hardy-identities} (and therefore, Theorem \ref{thm:main-theorem2}) will be provided in Section \ref{section:mean-distance-function-results}.

As an application of Theorem \ref{thm:mean-bessel-pair-hardy-identities}, we apply these inequalities to obtain lower bounds on the smallest eigenvalue for minus the Dirichlet Laplacian on a given domain $\Omega$.
The first bound of this type we consider was obtained by Davies in \cite{D} (see also the references therein) where they obtained
\begin{equation}
	H_{0}\geq\frac{N}{4\mu^{2}}
	\label{eq:introduction-spectral-lower-bound-by-davies}
\end{equation}
in the sense of bilinear forms, where where $H_{0}$ is minus the Dirichlet Laplacian on $\Omega$ and
\begin{equation}
	\mu = \sqrt{N} \sup\left\{ d_{\Omega,M,2}(x): x \in \Omega \right\}
	\label{eq:quasi-inradius}
\end{equation}
is the quasi-inradius of $\Omega$.
Note that the mean distance function $m$ used in \cite{D} is such that $m(x) = \sqrt{N} d_{\Omega,M,2}(x)$.

Before we continue the discussion, we bring to attention the fact that, since we are making use of $p$-Bessel pairs to obtain these results, and since $V$ and $W$ are defined on an interval $(0,R)$, we need $V(\rho_{\nu}(x))$ and $W(\rho_{\nu}(x))$ to be well-defined almost everywhere.
This leads us to introduce the following new type of diameter:
\[
D_{\oo}(\Omega)=\text{essential supremum of lengths of all lines contained in $\Omega$};
\]
see \eqref{eq:essential-diameter-definition} for a precise definition and see Section \ref{section:mean-distance-function-results} for further discussion about $D_{\oo}(\Omega)$.
Then, assuming $(V,W)$ is a $p$-Bessel pair on $(0,D_{\oo}(\Omega)/2)$, we have that $V(\rho_{\nu})$ and $W(\rho_{\nu})$ are defined almost everywhere on $\Omega$.

We use our results to provide a new lower bound for $H_{0}$ in terms of $R_{\oo}(\Omega)$ that greatly improves \eqref{eq:introduction-spectral-lower-bound-by-davies}; in fact, our results show that \eqref{eq:introduction-spectral-lower-bound-by-davies} is never sharp for any domain $\Omega$.
In particular, if $H_{0}$ denotes minus the Dirichlet Laplacian on $\Omega$, then we improve \eqref{eq:introduction-spectral-lower-bound-by-davies} in the spirit of Brezis and Marcus \cite{BM} by obtaining the following result
\begin{theoremintro}[Theorem \ref{cor:spectral-lower-bound-brezis-marcus-type}]
	Let $N,\Omega$ be as above. There holds
	\[
	H_{0}\geq\frac{N}{4\mu^{2}}+\frac{4N\la_{0}^{2}}{D_{\oo}(\Omega)^{2}}
	\]
	in the sense of bilinear forms.
	\label{thm:main-theorem3}
\end{theoremintro}

Note that, to obtain the lower bound on the smallest eigenvalue of $H_{0}$, one can simply apply these bounds to the Rayleigh quotient for $H_{0}$.

Theorem \ref{cor:spectral-lower-bound-brezis-marcus-type} (Theorem \ref{thm:main-theorem3}) will be proved in Section \ref{section:mean-distance-function-results}.

The last theorem we emphasize uses the Hardy-Sobolev-Maz'ya inequality on certain geometric domains in conjunction with our Hardy identities to establish Sobolev inequalities for a family of (generally rough) Riemannian metrics that degenerate at the boundary of the domain.
This result and related results are contained and proved in Section \ref{section:sobolev-inequality}.
We say that $\Omega$ supports a Hardy-Sobolev-Maz'ya inequality provided there exists a positive constant $C_{HSM}$ such that
\begin{equation}
  \int\limits_{\Omega}|\grad{u}|^{2}dx-\frac{1}{4}\int\limits_{\Omega}\frac{|u|^{2}}{d_{\Omega}^{2}}dx\geq{C_{HSM}}\left( \int\limits_{\Omega}|u|^{\frac{2N}{N-2}}dx \right)^{\frac{N-2}{N}},\quad{u\in{C_{0}^{\oo}(\Omega)}}.
  \label{eq:hsm}
\end{equation}

\begin{theoremintro}[Theorem \ref{thm:sobolev-inequality-theorem}]
  \label{thm:main-theorem4}
  Suppose $\Omega$ supports the Hardy-Sobolev-Maz'ya inequality \eqref{eq:hsm} and $\Delta d_{\Omega} \geq 0$ in the sense of distributions.
  If $\Omega$ is endowed with the (generally rough) metric $g = d_{\Omega}^{\frac{2}{N-2}}|dx|^{2}$, then the following Sobolev inequality holds on $(\Omega,g)$:
  \[
    C_{HSM}\left( \int\limits_{\Omega} |v|^{2^{*}}dV_{g} \right)^{\frac{1}{2^{*}}} \leq  \left( \int\limits_{\Omega} |\grad_{g}v|_{g}^{2} dV_{g} \right)^{\frac{1}{2}}
  \]
  for all $v \in C_{0}^{\oo}(\Omega)$.
  %such that $d_{\Omega}^{\frac{1}{2}}v \in C_{0}^{1,1}(\Omega)$.
  Here, $2^{*} =  2N/(N-2)$.
\end{theoremintro}

The proof of Theorem \ref{thm:sobolev-inequality-theorem} (Theorem \ref{thm:main-theorem4}) will be given in Section \ref{section:sobolev-inequality}.

Examples of such domains include halfspaces and domains of the form $\Omega = \R^{N} \setminus \overline{U}$, where $U$ is a weakly mean convex domain containing the origin and $d_{\Omega}$ satisfies \eqref{eq:super-harm-for-unbounded-domains}; e.g., $\Omega = \R^{N} \setminus \overline{B(0,1)}$ and the exterior of the solid cylinder $\R \times B(0,1) \subset \R^{N+1}$, $N>1$.

\medskip

%The three aforementioned results and related results are discussed in the discussion below.
Before we continue we point out that this paper is a continuation of the program initiated by Lam, Lu and Zhang in \cite{LLZ,LLZ20} of using Bessel pairs to obtain Hardy identities.
Since then, the present authors along with other collaborators have made contributions in this direction for several settings.
See \cite{FLL} for Hardy identities in the group setting and see \cite{FLL-JFA} and \cite{FLLM} for Hardy identities in hyperbolic spaces, the negatively curved and Riemannian manifolds, and \cite{Wang} for Hardy's identities in the Dunkl setting. The $L^p$ Hardy identities and inequalities
have also been established with $p$-Bessel pairs in \cite{DLL}. These type of $L^2$ Hardy identities using the $2$-Bessel pairs have also been extended to the second order case on hyperbolic spaces in \cite{BGR-CVPDE}. We also refer the interested reader to \cite{BFT04, BGGP20, DPP17}, for instance, for some related results.

%\textbf{TODO: and cite the rest as given in Carnot group paper}.
It is also important to mention that formulating Hardy inequalities in terms of Bessel pairs was first introduced by Ghoussoub and Moradifam.
Indeed, Ghoussoub and Moradifam introduced in \cite{GM} (see also their book \cite{GM2}) the notion of Bessel pairs to enhance and unify several improved Hardy type inequalities in the literature for the Euclidean setting.
They established the following important characterization theorem.

\begin{theoremalph}
  Assume $N\geq1$, let $0<R\leq\oo$, and let $V$ and $W$ be $C^{1}$ functions on $(0,R)$ such that $\int\limits_{0}^{R}\frac{1}{r^{N-1}V(r)}dr=\oo$ and $\int\limits_{0}^{R}r^{N-1}V(r)dr<\oo$.
  If $(r^{N-1}V,r^{N-1}W)$ is a Bessel pair on $(0,R)$, then for all $u\in{C_{0}^{\oo}(B_{R}(0))}$, there holds
  \begin{equation}
    \int\limits_{B_{R}(0)}V(|x|)|\grad{u}|^{2}\geq\int\limits_{B_{R}(0)}W(|x|)|u|^{2}dx.
    \label{eq:ghoussob-moradifam-hardy}
  \end{equation}
  Also, if \eqref{eq:ghoussob-moradifam-hardy} holds for all $u\in{C_{0}^{\oo}(B_{R}(0))}$, then $(r^{N-1}V,r^{N-1}cW)$ is a Bessel pair on $(0,R)$ for some $c>0$.
\end{theoremalph}

\medskip

We now recall results in the literature and point out connections with our results.
The reader should refer to the Main Results and Proofs section (Section \ref{main-results}) for precise statements of our results.
We also point out that any result labeled with a Latin letter (e.g., Theorem \ref{thm:brezis-marcus-lewis-li-li}) indicates a result known in the literature.

We begin by pointing out that, when studying the classical Hardy inequality \eqref{eq:hardy-inequality-general-domain}, we are primarily interested in two cases.
Namely, when $d=d_{\Omega}$ and $\Omega$ is weakly mean convex, or when $d = d_{\Omega,M,p}$ and $\Omega$ is assumed to be a general proper subdomain of $\R^{N}$ without any assumed boundary regularity.
(In the latter case, we elect in accordance with \cite{frank2020consequences} to call \eqref{eq:hardy-inequality-general-domain} the Davies-Hardy-Tidblom inequality.)
In the general domain case with $d=d_{\Omega,M,p}$, Davies introduced $d_{\Omega,M,2}$ and established \eqref{eq:hardy-inequality-general-domain} for $p=2$ in \cite{D}, and Tidblom extended this result to any $p$ satisfying $1<p<\oo$ in \cite{Ti} and introduced $d_{\Omega,M,p}$.
They both showed that $c_{p}$ can be taken to be $\left( \frac{p-1}{p} \right)^{p}$.
In case $d=d_{\Omega}$, then inequality \eqref{eq:hardy-inequality-general-domain} holds for all domains $\Omega$ only when $N<p<\oo$; this was established by Lewis in \cite{L}.
In case $1<p\leq{N}$, the validity of \eqref{eq:hardy-inequality-general-domain} for a general domain $\Omega$ and finding the corresponding sharp constant $c_{p}$ is a more delicate problem which depends on the geometry and regularity of $\p\Omega$.

For instance, in case $\Omega$ is a convex domain which is $C^{2}$ in a neighborhood of a boundary point, then \eqref{eq:hardy-inequality-general-domain} is sharp with $c_{p}=\left( \frac{p-1}{p} \right)^{p}$; see \cite{MMP,MS}.
When $\Omega$ is convex and $C^{2}$ and $p=2$, Brezis and Marcus established in \cite{BM} improvements of \eqref{eq:hardy-inequality-general-domain} by adding nontrivial nonnegative terms to the right side of \eqref{eq:hardy-inequality-general-domain}.
More precisely, Brezis and Marcus established the following theorem.
\begin{theoremalph}
  \label{thm:brezis-marcus}
  For every smooth domain of class $C^{2}$, there exists $\la(\Omega)\in\R$ such that
  \[
    \int\limits_{\Omega}|\grad{u}|^{2}dx-\frac{1}{4}\int\limits_{\Omega}\frac{|u|^{2}}{d_{\Omega}^{2}}dx\geq\la(\Omega)\int\limits_{\Omega}|u|^{2}dx,\quad{u\in{C_{0}^{\oo}(\Omega)}}.
  \]
  Moreover, for convex domains, there holds
  \[
    \la(\Omega)\geq\frac{1}{4\diam(\Omega)^{2}}.
  \]
\end{theoremalph}
Brezis and Marcus asked if $\la(\Omega)$ can be replaced by a constant depending on the volume $|\Omega|$ alone, and this was answered affirmatively in \cite{HOL,Ti}; in \cite{Ti}, Tidblom extended Theorem \ref{thm:brezis-marcus} to $1<p<\oo$.
Moreover, we mention that Avkhadiev and Wirths determined in \cite{AW} a lower bound of $\la(\Omega)$ in terms of the inradius for convex domains (see Theorem \ref{thm:avk-wirths-theorem}).
When $\Omega$ is assumed to be weakly mean convex, then the results of Brezis, Marcus and Tidblom were extended to this setting by Lewis, Li and Li in \cite{LLL}, and they also showed that \eqref{eq:hardy-inequality-general-domain} is sharp with $c_{p}=\left( \frac{p-1}{p} \right)^{p}$ being the best constant.
We point out that Lewis, Li and Li also provided examples (based on \cite{MMP}) of domains which are not weakly mean convex and which are such that \eqref{eq:hardy-inequality-general-domain} fails for any $c_{n+1}>0$ when $p=n+1$ (exterior of a ball) or when $c_{p}=\left( \frac{p-1}{p} \right)^{p}$ (ellipsoidal shell) when $d=d_{\Omega}$.
We collect the positive results into the following theorem stated for the case of weakly mean convex domains (i.e., as presented in \cite{LLL}).
Beware that \cite{LLL} uses the opposite sign convention for mean curvature.

\begin{theoremalph}
  \label{thm:brezis-marcus-lewis-li-li}
  Let $\Omega\subset\R^{N}$, $N\geq2$, be a weakly mean convex domain and let $1<p<\oo$.
  Then, for $u\in{C_{0}^{\oo}(\Omega)}$, there holds
  \[
    \int\limits_{\Omega}|\grad{u}|^{p}dx\geq\left( \frac{p-1}{p} \right)^{p}\int\limits_{\Omega}\frac{|u|^{p}}{d_{\Omega}^{p}}dx+\la(N,p,\Omega)\int\limits_{\Omega}|u|^{p}dx,
  \]
  where
  \[
    \la(N,p,\Omega)=\left( \frac{p-1}{p} \right)^{p-1}\inf_{\Omega\setminus{S}}\frac{-\Delta{d_{\Omega}}}{d_{\Omega}^{p-1}}\geq\frac{p}{(N-1)^{p-1}}|H_{0}|^{p},
  \]
  $S$ is the singular set of $d_{\Omega}$ and $H_{0}$ is the supremum of the mean curvature of $\p\Omega$.
\end{theoremalph}

Now, since $\la(N,p,\Omega)\geq0$ for weakly mean convex domains, we have that Theorem \ref{thm:brezis-marcus-lewis-li-li} implies the following sharp $L^{p}$-Hardy inequality for weakly mean convex domains:
\begin{equation}
  \label{eq:introduction-weakly-mean-convex-hardy-inequality}
  \int\limits_{\Omega}|\grad{u}|^{p}dx\geq\left( \frac{p-1}{p} \right)^{p}\int\limits_{\Omega}\frac{|u|^{p}}{d_{\Omega}^{p}}dx.
\end{equation}
% \begin{theoremalph}
%   For any convex domain $\Omega\subset\R^{N}$, $1<p<\oo$, and $u\in{C_{0}^{\oo}(\Omega)}$, we have
%   \begin{equation}
%     \int\limits_{\Omega}|\grad{u}|^{p}dx\geq\left( \frac{p-1}{p} \right)^{p}\int\limits_{\Omega}\frac{|u(x)|^{p}}{d_{\Omega}^{p}(x)}dx+\frac{a(p,N)}{|\Omega|^{\frac{p}{N}}}\int\limits_{\Omega}|u(x)|^{p}dx,
%     \label{}
%   \end{equation}
%   where
%   \[
%     a(p,n)=\frac{(p-1)^{p+1}}{p^{p}}\left( \frac{\omega_{N-1}}{N} \right)^{\frac{p}{N}}\frac{\sqrt{\pi}\Gamma\left( \frac{N+p}{2} \right)}{\Gamma\left( \frac{p+1}{2} \right)\Gamma\left( \frac{N}{2} \right)},
%   \]
%   where $\omega_{N-1}$ is the surface area of the $S^{N-1}$.
%   \label{thm:brezis-marcus-type-inequality}
% \end{theoremalph}
In Theorem \ref{thm:bessel-pair-hardy-identity-general-domain}, we establish a family of Hardy identities in terms of $p$-Bessel pairs which imply, among other things, strong improvements of Theorem \ref{thm:brezis-marcus-lewis-li-li} and the Hardy inequality \eqref{eq:introduction-weakly-mean-convex-hardy-inequality} in the case that either $u$ is assumed to be compactly supported away from the singular set of $d_{\Omega}$, or in the case that the singular set of $d_{\Omega}$ is sufficiently small.
These identities take the form
\begin{align*}
  &\int\limits_{\Omega}V(d_{\Omega}(x))|\grad{u(x)}|^{p}dx-\int\limits_{\Omega}W(d_{\Omega}(x))|u(x)|^{p}dx\\
  &=\int\limits_{\Omega}V(d_{\Omega}(x))C_{p}\left( \grad{u(x)},\varphi(d_{\Omega}(x))\grad\left( \frac{u(x)}{\varphi(d_{\Omega}(x))} \right) \right)dx\\
%  &+\int\limits_{\Omega}V(d_{\Omega}(x))|u(x)|^{p}\varphi(d_{\Omega}(x))^{1-p}|\varphi'(d_{\Omega}(x))|^{p-2}\varphi'(d_{\Omega}(x)) \frac{\a-1}{d_{\Omega}(x)}dx\\
  &- \gen{\Delta d_{\Omega}, V\circ d_{\Omega} \left|\frac{\varphi'\circ d_{\Omega}}{\varphi\circ d_{\Omega}}\right|^{p-2}\frac{\varphi'\circ d_{\Omega}}{\varphi\circ d_{\Omega}} |u|^{p}}
\end{align*}
%where $1<p<\oo$, where $(r^{\a-1}V,r^{\a-1}W)$ is a $p$-Bessel pair for some $a\in\R$, where $\Delta d_{\Omega}$ is the distributional Laplacian of $d_{\Omega}$ (note that $d_{\Omega}$ is only Lipschitz),  and where
where $1<p<\oo$, where $(V,W)$ is a $p$-Bessel pair with a positive solution $\varphi$, where $\Delta d_{\Omega}$ is the distributional Laplacian of $d_{\Omega}$ (note that $d_{\Omega}$ is a priori only Lipschitz) and where
\[
  C_{p}(x,y)=|x|^{p}-|x-y|^{p}-p|x-y|^{p-2}(x-y)\cdot{y},\quad{x,y\in\R^{N}};
\]
see Section \ref{preliminaries} for more about this $C_{p}$ function and the distributional Laplacian of $d_{\Omega}$.
We point out that, when $\p\Omega$ is sufficiently regular, the distributional Laplacian $\Delta d_{\Omega}$ decomposes into two distributions $\Delta_{G}d_{\Omega}$ and $\Delta_{\Sigma} d_{\Omega}$, where $\Delta_{G} d_{\Omega}$ is the function obtained by taking the Laplacian of $d_{\Omega}$ restricted to the good set $G$ (the largest open set where $d_{\Omega}$ is $C^{2}$), and $\Delta_{\Sigma}d_{\Omega}$ is a distribution given by an $(N-1)$-dimensional integral over the boundary $\p\Omega$.
We emphasize that the appearance of a lower dimensional integral in the difference
\[
  \int\limits_{\Omega}|\grad u(x)|^{p} dx - \int\limits_{\Omega}\frac{|u(x)|^{p}}{d_{\Omega}(x)^{p}}dx
\]
is a feature of Hardy inequalities with distance function to the boundary of a domain and that such a term is not generally present for Hardy inequalities in which the distance function is to a given point.
It is also worth mentioning that this lower dimensional integral vanishes for certain $p$-Bessel pairs, when $\Sigma$ has $(N-1)$-dimensional Hausdorff measure zero, when $d_{\Omega}$ is harmonic or when we restrict our attention to functions $u \in C_{0}^{\oo} (G)$.

%As an immediate application by using an appropriate $p$-Bessel pair $(V,W)$ (with $\a=1$), we compute the exact remainder terms of so that \eqref{eq:introduction-weakly-mean-convex-hardy-inequality} becomes the Hardy identity
By applying Theorem \ref{thm:bessel-pair-hardy-identity-general-domain} to an appropriate $p$-Bessel pair and to a weakly mean convex $\Omega$ with $C^{2,1}$ boundary, we obtain the following Hardy identities which generalize and improve \eqref{eq:introduction-weakly-mean-convex-hardy-inequality}:
\begin{equation}
  \begin{aligned}
    \int\limits_{\Omega}\frac{|\grad{u(x)}|^{p}}{d_{\Omega}(x)^{\la}}dx&-\bigg\vert\frac{p+\la-1}{p}\bigg\vert^{p}\int\limits_{\Omega}\frac{|u(x)|^{p}}{d_{\Omega}(x)^{\la+p}}\\
    &=\int\limits_{\Omega}\frac{1}{d_{\Omega}(x)^{\la}}C_{p}\left( \grad{u(x)},d_{\Omega}(x)^{\frac{p+\la}{p}}\grad\left( d_{\Omega}(x)^{-\frac{\la+p}{p}}u(x) \right) \right)dx\\
    &-\bigg\vert\frac{p+\la-1}{p}\bigg\vert^{p-1}\gen{\Delta d_{\Omega},|u|^{p} d_{\Omega}^{1-\la-p}},
  \end{aligned}
  \label{eq:weakly-mean-convex-lp-hardy-identity-general-powers}
\end{equation}
where $1<p<\oo$ and $\la\in\R$ are such that $p+\la-1>0$.
This is the content of Corollary \ref{cor:weighted-hardy-identity-for-domains}.
We also point out that the distribution $\Delta d_{\Omega}$ has an explicit geometric expression in terms of curvature and geodesic normal coordinates.
Also, observing that, when $d_{\Omega}$ is distributionally superharmonic on $\Omega$, i.e., when $-\Delta{d_{\Omega}}\geq0$ as a distribution, then
\begin{align*}
  0\leq\int\limits_{\Omega}\frac{1}{d_{\Omega}(x)^{\la}}&C_{p}\left( \grad{u(x)},d_{\Omega}(x)^{\frac{p+\la}{p}}\grad\left( d_{\Omega}(x)^{-\frac{\la+p}{p}}u(x) \right) \right)dx\\
  &-\bigg\vert\frac{p+\la-1}{p}\bigg\vert^{p-1}\gen{\Delta d_{\Omega},|u|^{p} d_{\Omega}^{1-\la-p}},
\end{align*}
it is clear that \eqref{eq:weakly-mean-convex-lp-hardy-identity-general-powers} with $\la=0$ is a strong improvement of \eqref{eq:hardy-inequality-general-domain} for weakly mean convex domains.
In fact, by \cite[Theorem 1.7]{LLL}, the superharmonicity of $d_{\Omega}$ in $\Omega$ is equivalent to $\Omega$ being weakly mean convex.

We note that, in case $p=2$ and $\la=0$, identity \eqref{eq:weakly-mean-convex-lp-hardy-identity-general-powers} recovers Lemma 4.1 in \cite{LLL}.
Indeed, to see this, note
\begin{align*}
  \int\limits_{\Omega}\bigg\vert\grad{u(x)}-\frac{u(x)\grad{d_{\Omega}(x)}}{2d_{\Omega}(x)}\bigg\vert^{2}dx&=\int\limits_{\Omega}d_{\Omega}(x)\bigg\vert\grad\left( \frac{u(x)}{d_{\Omega}(x)^{\frac{1}{2}}} \right)\bigg\vert^{2}dx\\
  &=\int\limits_{\Omega}C_{2}\left( \grad{u(x)},d_{\Omega}(x)^{\frac{1}{2}}\grad\left( \frac{u(x)}{d_{\Omega}(x)^{\frac{1}{2}}} \right) \right)dx.
\end{align*}
On the other hand, the identities for general $p$ with $1<p<\oo$ are not given in \cite{LLL} and therefore \eqref{eq:weakly-mean-convex-lp-hardy-identity-general-powers} is new for $p\neq2$.

In addition to \eqref{eq:weakly-mean-convex-lp-hardy-identity-general-powers}, in Theorem \ref{thm:bessel-pair-hardy-identity-general-domain} we also establish Hardy identities in which $\grad$ is replaced by $\grad{d_{\Omega}}\cdot\grad$, i.e., we establish Hardy identities of the following form:
\begin{align*}
  &\int\limits_{\Omega}{V(d_{\Omega}(x))}|\nabla{d_{\Omega}(x)}\cdot\nabla{u(x)}|^{p}dx -  \int\limits_{\Omega}W\left( d_{\Omega}(x) \right) |u(x)|^{p} dx \\
  &=\int\limits_{\Omega}{V(d_{\Omega}(x))}C_{p}\left(\nabla{d_{\Omega}(x)}\cdot{u(x)},\varphi(d_{\Omega}(x))\nabla{d_{\Omega}(x)}\cdot\nabla\left( \frac{u(x)}{\varphi(d_{\Omega}(x))} \right)\right)dx\\
    %&+ \int\limits_{\Omega} V(d_{\Omega}(x)) |u(x)|^{p} \varphi(d_{\Omega}(x))^{1-p} |\varphi'(d_{\Omega}(x))|^{p-2} \varphi'(d_{\Omega}(x))\left[ \Delta d_{\Omega}(x) - \frac{\a-1}{d_{\Omega}(x)} \right] dx.
  &- \gen{\Delta d_{\Omega},V \circ d_{\Omega}\left|\frac{\varphi' \circ d_{\Omega}}{\varphi \circ d_{\Omega}}\right|^{p-2}\frac{\varphi' \circ d_{\Omega}}{\varphi \circ d_{\Omega}} |u|^{p}}.
\end{align*}
Note that $|\grad{d_{\Omega}}\cdot\grad{u}|\leq|\grad{u}|$, and so our identities imply geometrically significant refinements of \eqref{eq:introduction-weakly-mean-convex-hardy-inequality}.
Moreover, $\grad d_{\Omega}$ also has a (almost everywhere) geometric expression and so the operator $\grad d_{\Omega} \cdot $ is more or less explicit; see Theorem \ref{thm:gradient-of-distance-function}.

\medskip

In \cite{AW}, Avkhadiev and Wirths studied the best constant $c(N)$, depending on the dimension $N$ only, in the inequality
\begin{equation}
  \int\limits_{\Omega}|\grad{u(x)}|^{2}dx\geq\frac{1}{4}\int\limits_{\Omega}\frac{|u(x)|^{2}}{d_{\Omega}(x)}dx+\frac{c(N)}{\d_{0}^{2}}\int\limits_{\Omega}|u(x)|^{2}dx,
  \label{eq:avk-wirth-inequality-with-general-c}
\end{equation}
where $u\in{C_{0}^{\oo}(\Omega)}$, $\d_{0}=\d_{0}(\Omega)$ is the inradius of $\Omega$ and $\Omega$ is restricted to the class of open convex domains with $\d_{0}<\oo$.
(See the discussion below about the Hardy inequalities of Brezis-Marcus-type for more information about inequalities of this type.)
We state the results of Avkhadiev and Wirths in the following theorem.

\begin{theoremalph}
  \label{thm:avk-wirths-theorem}
  Let $\Omega$ be an open and convex set in $\R^{N}$.
  If the inradius $\d_{0}=\d_{0}(\Omega)$ is finite, then
  \begin{equation}
    \int\limits_{\Omega}|\grad{u}|^{2}dx\geq\frac{1}{4}\int\limits_{\Omega}\frac{|u|^{2}}{d_{\Omega}^{2}}dx+\frac{\la_{0}^{2}}{\d_{0}^{2}}\int\limits_{\Omega}|u|^{2}dx,\quad{u\in{C_{0}^{\oo}(\Omega)}},
    \label{eq:avk-wirths-inequality}
  \end{equation}
  where $\la_{0}$ is the Lamb constant, i.e., if $J_{0}$ the Bessel function \eqref{eq:bessel-function}, then $\la_{0}$ is the first positive zero of
  \begin{equation}
    J_{0}(r)+2rJ_{0}'(r)=0.
    \label{eq:lambs-equation}
  \end{equation}
  Moreover, the constant $\la_{0}^{2}$ is the largest constant so that \eqref{eq:avk-wirths-inequality} holds for all such convex domains and it is attained by domains of the form $(a,b)\times\R^{N-1}$ with $N\geq1$ and their linear transformations.
\end{theoremalph}

Our methods allow us to find the exact remainder terms for \eqref{eq:avk-wirths-inequality} and extend Theorem \ref{thm:avk-wirths-theorem} to weakly mean convex domains.
%, provided we restrict to the class of domains with $C^{2,1}$ boundary.
In particular, we prove in Corollary \ref{cor:avk-wirths-improvmenet} the following Hardy identity (and its weighted versions):

\begin{equation}
  \begin{aligned}
    \int\limits_{\Omega}|\grad{u(x)}|^{2}dx&-\frac{1}{4}\int\limits_{\Omega}\frac{|u(x)|^{2}}{d_{\Omega}(x)^{2}}dx=\frac{\la_{0}^{2}}{\d_{0}^{2}}\int\limits_{\Omega}|u(x)|^{2}dx\\
    &+\int\limits_{\Omega}\frac{|J_{0}(\frac{\la_{0}}{\d_{0}}d_{\Omega}(x))|^{2}}{d_{\Omega}(x)^{-2}}\bigg\vert\grad\left( \frac{d_{\Omega}(x)^{-\frac{1}{2}}u(x)}{J_{0}(\frac{\la_{0}}{\d_{0}}d_{\Omega}(x))} \right)\bigg\vert^{2}dx\\
    &-\frac{1}{2}\gen{\Delta d_{\Omega},u^{2}\frac{1}{d_{\Omega}} \left[ \frac{1}{2}\frac{1}{d_{\Omega}}+\frac{\la_{0}}{\d_{0}}\frac{J_{0}'\left( \frac{\la_{0}}{\d_{0}}d_{\Omega} \right)}{J_{0}\left( \frac{\la_{0}}{\d_{0}}d_{\Omega} \right)} \right]},
  \end{aligned},
  \label{eq:introduction-identity-for-avk-wirths}
\end{equation}
where $\d_{0}<\oo$ is the inradius of $\Omega$ and $u \in C_{0}^{\oo}(\Omega)$.
As before, when $\Omega$ is weakly mean convex, then the remainder term is nonnegative (see proof of Corollary \ref{cor:avk-wirths-improvmenet}).

In fact, this identity gives a new interpretation and reasoning as to why the constant $\la_{0}^{2}$ ought to appear.
Indeed, $(0,\la_{0})$ is the largest interval of the form $(0,a)$, $a\in\R_{>0}$, for which
\[
  \frac{1}{2}\frac{1}{d_{\Omega}}+\frac{\la_{0}}{\d_{0}}\frac{J_{0}'\left( \frac{\la_{0}}{\d_{0}}d_{\Omega} \right)}{J_{0}\left( \frac{\la_{0}}{\d_{0}}d_{\Omega} \right)}
\]
is nonnegative.
With weak mean convexity, this forces the remainder term to be nonnegative, and, on any larger interval, the remainder term will change sign.

We also mention that identity \eqref{eq:introduction-identity-for-avk-wirths} gives a direct means for determining the existence of extremizers and their functional form since, if \eqref{eq:avk-wirths-inequality} were in fact an identity, then the remainder term must be zero:
\begin{align*}
  0=&\int\limits_{\Omega}\frac{|J_{0}(\frac{\la_{0}}{\d_{0}}d_{\Omega}(x))|^{2}}{d_{\Omega}(x)^{-2}}\bigg\vert\grad\left( \frac{d_{\Omega}(x)^{-\frac{1}{2}}u(x)}{J_{0}(\frac{\la_{0}}{\d_{0}}d_{\Omega}(x))} \right)\bigg\vert^{2}dx\\
  &-\frac{1}{2}\int\limits_{\Omega}\frac{|u(x)|^{2}}{d_{\Omega}(x)}\Delta{d_{\Omega}(x)}\left[ \frac{1}{2}\frac{1}{d_{\Omega}(x)}+\frac{\la_{0}}{\d_{0}}\frac{J_{0}'\left( \frac{\la_{0}}{\d_{0}}d_{\Omega}(x) \right)}{J_{0}\left( \frac{\la_{0}}{\d_{0}}d_{\Omega}(x) \right)} \right]dx.
\end{align*}
This implies
\[
  \bigg\vert\grad\left( \frac{d_{\Omega}(x)^{-\frac{1}{2}}u(x)}{J_{0}(\frac{\la_{0}}{R}d_{\Omega}(x))} \right)\bigg\vert^{2}=0
\]
and so
\[
  u(x)=cd_{\Omega}(x)^{\frac{1}{2}}J_{0}\left( \frac{\la_{0}}{R}d_{\Omega}(x) \right)
\]
for some $c\in\R$.
Next observe that, since $-\Delta d_{\Omega} \geq0$, this implies that
\[
  \gen{\Delta d_{\Omega},u^{2}\frac{1}{d_{\Omega}} \left[ \frac{1}{2}\frac{1}{d_{\Omega}}+\frac{\la_{0}}{\d_{0}}\frac{J_{0}'\left( \frac{\la_{0}}{\d_{0}}d_{\Omega} \right)}{J_{0}\left( \frac{\la_{0}}{\d_{0}}d_{\Omega} \right)} \right]}=0.
\]
%If $\Omega$ is $C^{2}$, then, by Theorem \ref{thm:main-theorem2}, we conclude $\Omega$ is a strip or halfspace and hence $u$ is not integrable on $\Omega$ thereby establishing the nonexistence of extremizers for \eqref{eq:avk-wirths-inequality} for \textit{every} $C^{2}$-convex domain.
If $\Omega$ is $C^{2}$, then, by Corollary \ref{cor:bounded-implies-not-bounded}, we conclude $\Omega$ is a strip or halfspace and hence $u$ is not integrable on $\Omega$ thereby establishing the nonexistence of extremizers for \eqref{eq:avk-wirths-inequality} for \textit{every} $C^{2}$-convex domain.

%\textbf{TODO: can we weaken to $C^{1}$? Lipschitz would be great since convex implies lipschitz, but not sure if flatness results hold. Seems we get $d_{\Omega}(singular\,set) \equiv \d_{0}$ or $\Delta d_{\Omega} = 0$ as a distribution (hence smooth hence $\Omega$ unbounded by Theorem H). Both should imply nonexistence of extremizer}

%TODO: idk why I said this: We also obtain improvements of \eqref{eq:avk-wirths-inequality} in case $\Omega$ is a general convex domain with finite inradius; this is the content of Corollary \textbf{TODO:cite}.

\medskip

We point out that, to prove Theorem \ref{thm:brezis-marcus} for convex domains and $1<p<\oo$, Tidblom first established the following inequality in the spirit of Davies \cite{D} and Brezis and Marcus \cite{BM} for general domains.

\begin{theoremalph}
  \label{thm:davies-tidblom-theorem}
  Let $\Omega\subsetneq\R^{N}$ be an open domain.
  Then, for $u\in{C_{0}^{\oo}(\Omega})$, there holds
  \begin{equation}
    \begin{aligned}
      \int\limits_{\Omega}|\grad{u}|^{p}dx\geq\left( \frac{p-1}{p} \right)^{p}\int\limits_{\Omega}\frac{|u(x)|^{p}}{d_{\Omega,M,p}(x)^{p}}dx+a(p,N)\int\limits_{\Omega}\frac{|u(x)|^{p}}{|\Omega_{x}|^{\frac{p}{N}}}dx,
    \end{aligned}
    \label{eq:tidblom-davies-inequality}
  \end{equation}
  where
  \begin{align*}
    \Omega_{x}&=\left\{ y\in\Omega:x+t(y-x)\in\Omega,\forall{t\in[0,1]} \right\}\\
    a(p,n)&=\frac{(p-1)^{p+1}}{p^{p}}\left( \frac{\omega_{N-1}}{N} \right)^{\frac{p}{N}}\frac{\sqrt{\pi}\Gamma\left( \frac{N+p}{2} \right)}{\Gamma\left( \frac{p+1}{2} \right)\Gamma\left( \frac{N}{2} \right)}.
  \end{align*}
\end{theoremalph}

Theorem \ref{thm:davies-tidblom-theorem} is of course an extension and improvement (in the spirit of Brezis and Marcus) of the result of Davies (i.e., \eqref{eq:hardy-inequality-general-domain} with $p=2$ and $d=d_{\Omega,M,2}$).
In the same spirit, we establish the following identities and inequalities in terms of $p$-Bessel pairs:
\begin{align*}
  \int\limits_{\Omega}& \tilde{V}_{M,p}(x) |\grad u(x)|^{p}dx  = \int\limits_{\Omega} W_{M,p}(x) |u(x)|^{p} dx +2\mathcal{S}_{\Omega}\left[ V \circ \rho_{\nu} \left| \frac{\varphi' \circ \rho_{\nu}}{\varphi \circ \rho_{\nu}} \right|^{p-2}\frac{\varphi' \circ \rho_{\nu}}{\varphi \circ \rho_{\nu}} |u|^{p} \right]\\
  &+ \int\limits_{S^{N-1}}\int\limits_{\Omega} V(\rho_{\nu}(x)) C_{p}\left( \p_{\nu}u(x), \varphi(\rho_{\nu}(x)) \p_{\nu} \left( \frac{u(x)}{\varphi(\rho_{\nu}(x))} \right) \right) dx d\sigma(\nu).
\end{align*}
and, using $V_{M,p}\geq\tilde{V}_{M,p}$, there holds
\begin{align*}
  \int\limits_{\Omega}& V_{M,p}(x) |\grad u(x)|^{p}dx  \geq \int\limits_{\Omega} W_{M,p}(x) |u(x)|^{p} dx +2\mathcal{S}_{\Omega}\left[ V \circ \rho_{\nu} \left| \frac{\varphi' \circ \rho_{\nu}}{\varphi \circ \rho_{\nu}} \right|^{p-2}\frac{\varphi' \circ \rho_{\nu}}{\varphi \circ \rho_{\nu}} |u|^{p} \right]\\
  &+ \int\limits_{S^{N-1}}\int\limits_{\Omega} V(\rho_{\nu}(x)) C_{p}\left( \p_{\nu}u(x), \varphi(\rho_{\nu}(x)) \p_{\nu} \left( \frac{u(x)}{\varphi(\rho_{\nu}(x))} \right) \right) dx d\sigma(\nu),
\end{align*}
where $\mathcal{S}_{\Omega}[\cdot]$ is a certain mean functional that replaces the distributional Laplacian seen above; see \eqref{eq:spherical-skeletal-mean} for the definition of $\mathcal{S}_{\Omega}[\cdot]$.
This is the content of Theorem \ref{thm:mean-bessel-pair-hardy-identities}.
We remark that Davies' inequality was recently used by Frank and Larson in \cite{frank2020consequences} to provide new and short proofs for results of Lieb \cite{Li} and Rozenbljum \cite{R}.

%\textbf{TODO: maybe talk about how the Davies method can be applied to any convex domain without assuming smoothness of boundary, but this does not produce exact remainder terms?}

\medskip

%In the following two subsections, we will recall some results relevant to this paper as well as list out our main results.
%Section 2 will relevant preliminaries, and the remaining sections will contain proofs.

%\subsection{Mean Convex Domains}

We conclude this discussion with a mention of results pertaining to Hardy-Sobolev-Maz'ya type inequalities, as well as more precise statements of inequalities in the spirit of Brezis and Marcus \cite{BM}.
We also point out connections to our main results concerning Sobolev inequalities on certain noncompact noncomplete Riemannian manifolds.
To begin, we note that there has recently been much interest in improving the Hardy inequality by adding positive remainder terms to the right hand side of \eqref{eq:hardy-inequality-general-domain}.
One of the first results in this direction was obtained by Maz'ya in \cite{M} where it was shown that, for those  smooth functions on the half-space $\R_{+}^{N}=\left\{ x\in\R^{N}:x_{N>0} \right\}$, $N\geq3$, that vanish on the boundary $x_{N}=0$, there holds the sharp Hardy-Sobolev-Maz'ya inequality
\begin{equation}
  \int\limits_{\R_{+}^{N}}|\grad{u}|^{2}dx-\frac{1}{4}\int\limits_{\R_{+}^{N}}\frac{|u|^{2}}{x_{N}^{2}}dx\geq{C}\left( \int\limits_{\R_{+}^{N}}|u|^{\frac{2N}{N-2}}dx \right)^{\frac{N-2}{N}}.
  \label{eq:hsm-inequality}
\end{equation}
Note that \eqref{eq:hsm-inequality} combines both the Sobolev and Hardy inequalities into a single inequality.

The Hardy-Sobolev-Maz'ya inequality was later extended by Filippas, Maz'ya and Tertikas in \cite{FMT,FMT2} to domains $\Omega$ with finite inradius and satisfying $-\Delta d_{\Omega} \geq0$; i.e., under these assumptions, they proved (among more general results)
\begin{equation}
  \int\limits_{\Omega}|\grad{u}|^{2}dx-\frac{1}{4}\int\limits_{\Omega}\frac{|u|^{2}}{d_{\Omega}^{2}}dx\geq{C(\Omega)}\left( \int\limits_{\Omega}|u|^{\frac{2N}{N-2}}dx \right)^{\frac{N-2}{N}},\quad{u\in{C_{0}^{\oo}(\Omega)}}.
  \label{eq:hsm-inequality-on-convex-domains}
\end{equation}
In \cite{FMT2}, they asked if $C(\Omega)$ can be made independent of $\Omega$, and this was shown by Frank and Loss in \cite{FL} to be possible in case $\Omega$ is convex.
We emphasize that $-\Delta d_{\Omega} \geq 0$ is equivalent to convexity in two dimensions and weaker than convexity in higher dimensions (in fact, if $\Omega$ is $C^{2}$, then $-\Delta d_{\Omega} \geq0$ is equivalent to weak mean convexity \cite{LLL,P}).
However, \eqref{eq:hsm-inequality} holds for the domain $\R_{+}^{N}$ with infinite inradius, and so it is natural to ask if \eqref{eq:hsm-inequality-on-convex-domains} holds for more general unbounded domains.
This was answered affirmatively by Gkikas in \cite{G} provided $\Omega$ is an exterior domain satisfying
\begin{equation}
  -\Delta d_{\Omega} (x) + (N-1)\frac{\grad d_{\Omega} \cdot x}{|x|^{2}} \geq 0.
  \label{eq:super-harm-for-unbounded-domains}
\end{equation}
By exterior domain it is meant that $\Omega$ is the complement of relatively compact $C^{2}$ domain and $0 \notin \overline\Omega$.
This result is stated in the following theorem.
\begin{theoremalph}
  \label{thm:gkikas-hsm-inequality}
  Let $N\geq4$ and $\Omega$ be an exterior domain not containing the origin and satisfying condition \eqref{eq:super-harm-for-unbounded-domains}.
  Then the following inequality is valid:
  \begin{equation}
    \int\limits_{\Omega} |\grad u(x)|^{2} dx - \frac{1}{4}\int\limits_{\Omega} \frac{u(x)^{2}}{d_{\Omega}^{2}} \geq C \left( \int\limits_{\Omega} |u(x)|^{\frac{2N}{N-2}} dx \right)^{\frac{N-2}{N}},\quad u \in C_{0}^{\oo}(\Omega),
    \label{eq:gkikas-hsm-inequality}
  \end{equation}
  where $C>0$ depends only on $\Omega$ and $N$.
\end{theoremalph}

\medskip

%To conclude the introduction, we collect, for the convenience of the reader, the aforementioned applications of our results which were not stated in Theorems \ref{thm:main-theorem1}, \ref{thm:main-theorem2} and \ref{thm:main-theorem3} in the following list:
To conclude the introduction, we collect, for the convenience of the reader, the aforementioned applications of our results which were not stated in Theorems \ref{thm:main-theorem1} and \ref{thm:main-theorem4} in the following list:
\begin{enumerate}[(i)]
  \item our identities imply improvements, generalizations and refinements of the sharp Hardy inequalities of Brezis-Marcus, Tidblom, and Lewis-Li-Li \cite{BM,Ti,LLL} for weakly mean convex domains;
  \item our identities imply improvements, generalizations and refinements of the sharp Hardy inequalities of Avkhadiev-Wirths \cite{AW} as well as extend their inequalities to weakly mean convex domains and give a novel geometric interpretation of Lamb's constant $\la_{0}$ and the extremizers obtained by Avkhadiev-Wirths.
\end{enumerate}

\medskip

Our paper is organized as follows: In Section 2, we provide some results about the $p$-Bessel pairs and some geometric preliminaries. Main results, their consequences and their proofs will be given in Section 3: In Subsection \ref{subsec3.1}, we present the proofs for our first main result in this paper, namely, Theorem \ref{thm:one-dimensional-hardy-identities}, and its corollaries. The results from this subsection will be used for results pertaining to the mean distance function. In Subsection \ref{section:hardy-identities}, we state and prove Hardy identities for general domains and give the proofs for the second main result of this article, Theorem \ref{thm:bessel-pair-hardy-identity-general-domain}, and its consequences. Then, Theorems \ref{thm:one-dimensional-hardy-identities} and \ref{thm:bessel-pair-hardy-identity-general-domain} will be used to improve and generalize  Hardy inequalities for weakly mean convex domains (c.f. Theorems \ref{thm:brezis-marcus}, \ref{thm:brezis-marcus-lewis-li-li} and \ref{thm:avk-wirths-theorem}). The third main result, Theorem \ref{thm:mean-bessel-pair-hardy-identities}, will be proved and will be used to generalize Theorem \ref{thm:davies-tidblom-theorem} in Subsection \ref{section:mean-distance-function-results}. Theorem \ref{cor:spectral-lower-bound-brezis-marcus-type} will also be proved in this Subsection. Lastly, in Subsection \ref{section:sobolev-inequality}, Theorem \ref{thm:sobolev-inequality-theorem} combines the Hardy identities in Theorem \ref{thm:bessel-pair-hardy-identity-general-domain} and the Hardy-Sobolev-Maz'ya inequality of Theorem \ref{thm:gkikas-hsm-inequality} to obtain new Sobolev inequalities on noncomplete spaces whose Riemannian metric has low regularity.

\section{Preliminaries}
\label{preliminaries}

\subsection{$p$-Bessel Pairs}

The class of weights considered in this paper consist of pairs $(V,W)$ which satisfy some (generally nonlinear) ODE condition, as well as their spherical mean analogues. This notion of $p-$Bessel pairs was initially given in \cite{DLL}.
To be more precise, given $0<R \leq \oo$ and $1 < p < \oo$, a pair $(V,W)$ of positive $C^{1}$ functions is called a $p$-Bessel pair on an interval $(0,R)$ provided the nonlinear ODE
\begin{equation}
  \left( V|y'|^{p-2}y' \right)'+W|y|^{p-2}y=0
  \label{eq:p-bessel-pair-equation}
\end{equation}
has a positive solution $\varphi$ on $(0,R)$.
A positive solution $\varphi$ of \eqref{eq:p-bessel-pair-equation} for a given pair $(V,W)$ will simply be called a positive solution of the $p$-Bessel pair $(V,W)$.
In the special case $p=2$, the ODE \eqref{eq:p-bessel-pair-equation} is linear and takes the form
\[
  (Vy')'+Wy=0.
\]
this recovers the definition of $2-$Bessel pair in \cite{GM, GM2}.

We also point out that complications may arise when $1<p<2$ and $\varphi'$ has zeros in $(0,R)$.
We therefore adopt the convention that $\varphi'$ has constant sign on $(0,R)$ whenever $1 < p<2$.

Given a $p$-Bessel pair $(V,W)$, we introduce their spherical means as follows:
\begin{equation}
  \begin{aligned}
    \tilde{V}_{M,p}(x)&=\int\limits_{S^{N-1}}V(\rho_{\nu}(x))|\cos\left( \nu,\vec{e} \right)|^{p}d\sigma(\nu)\\
    V_{M,p}(x)&=\int\limits_{S^{N-1}}V(\rho_{\nu}(x))d\sigma(\nu)\\
    W_{M,p}(x)&=\int\limits_{S^{N-1}}W(\rho_{\nu}(x))d\sigma(\nu).
  \end{aligned}
  \label{eq:spherical-means-of-bessel-pairs}
\end{equation}
Here, $(a,b)$ denotes the angle between two vectors $a,b\in\R^{N}$.
Note that, if $V(r) = r^{\a}$, then $V_{M,p}$ is, up to a constant multiple, the mean distance function $d_{\Omega,M,\a}$ (it is clear that $d_{\Omega,M,\a}$ may be defined for $\a\in\R$ so long as $0 < \Gamma(\tfrac{N+x}{2})/\Gamma(\tfrac{x+1}{2}) < \oo$ when evaluated at $x=\a$).
Moreover, since $|\cos(t)|\leq1$, there holds $\tilde{V}_{M,p}\leq{V_{M,p}}$.

Next, define
\begin{align*}
  C_{p}(x,y)&=|x|^{p}-|x-y|^{p}-p|x-y|^{p-2}(x-y)\cdot{y}\\
  &=|x|^{p}+(p-1)|x-y|^{p}-p|x-y|^{p-2}(x-y)\cdot{x},
\end{align*}
where $x,y\in\R^{k}$ for some $k\geq1$.
Observe that $C_{2}(x,y)=|y|^{2}$.
Moreover, this function has the following two important properties.

\begin{lemmaalph}[\cite{DLL}]
  For all $x,y$, there holds $C_{p}(x,y)\geq0$, and $C_{p}(x,y)=0$ if and only if $y=0$.
  \label{lem:cp-properties}
\end{lemmaalph}
\noindent
The $C_{p}$ function will be used to explicitly write out the Hardy remainder terms as well as address the question of existence of extremizers.

\medskip
This section is concluded with examples of $p$-Bessel pairs which will be used.	

\begin{example}
  \label{B1}
  If $\la,p\in\R$, then $\left( r^{\lambda-1},\bigg|\frac{\lambda-p}{p}\bigg|^{p}r^{\lambda-1-p} \right)$ is a $p$-Bessel pair on $\left(  0,\infty\right)$ with $\varphi=r^{-\frac{\lambda-p}{p}}$.
\end{example}

\begin{lemma}
  Letting $\varphi(r)=r^{\frac{p-\a+\la}{p}}$, there holds
  \[
    |\varphi'(r)|^{p-2}\varphi'(r)\varphi^{1-p}(r)=\left\vert\frac{p-\a+\la}{p}\right\vert^{p-2}\frac{p-\a+\la}{p}r^{1-p}.
  \]
  \label{lem:varphi-stuff-computed}
\end{lemma}

\begin{proof}
  Using
  \begin{align*}
    \varphi(r)&=r^{\frac{p-\a+\la}{p}}\\
    \varphi'(r)&=\frac{p-\a+\la}{p}r^{\frac{-\a+\la}{p}},
  \end{align*}
  one finds
  \begin{align*}
    |\varphi'(r)|^{p-2}\varphi'(r)\varphi^{1-p}(r)&=\left\vert\frac{p-\a+\la}{p}\right\vert^{p-2}\frac{p-\a+\la}{p}r^{\frac{-\a+\la}{p}(p-2)}r^{\frac{-\a+\la}{p}}r^{\frac{p-\a+\la}{p}(1-p)}\\
    &=\left\vert\frac{p-\a+\la}{p}\right\vert^{p-2}\frac{p-\a+\la}{p}r^{1-p}.
  \end{align*}
\end{proof}

In preparation of the next example, we recall that the Bessel function of the first kind $J_{0}$ is the solution to the initial value problem
\begin{align*}
  \begin{cases}
    y''+\frac{1}{r}y'+y=0\\
    y(0)=1
  \end{cases}.
\end{align*}
Explicitly, $J_{0}$ has the following series expansion:
\begin{equation}
  J_{0}(r)=\sum_{j=0}^{\oo}\frac{(-1)^{j}}{(j!)^{2}}\left( \frac{r}{2} \right)^{2j}.
  \label{eq:bessel-function}
\end{equation}
Associated with the Bessel function $J_{0}$ is the so-called Lamb's constant $\la_{0}$, i.e., $\la_{0}$ is the first positive zero of
\[
  J_{0}(r)+2rJ_{0}'(r)=0.
\]
Approximately, $\la_{0} = 0.940\ldots$, and this constant is smaller than the first zero $z_{0} = 2.4048\ldots$ of $J_{0}$.
We emphasize that $\la_{0}$ is equivalently defined as the first positive zero of $(r^{\frac{1}{2}}J_{0}(r))'$, and that $(r^{\frac{1}{2}}J_{0}(r))'$ is positive on $(0,\la_{0})$ and negative on $(\la_{0},z_{0})$.

\begin{example}
  \label{B2}
  Let $z_{0}$ be the first positive zero of $J_{0}$ and let $0 < \Lambda \leq z_{0}$.
  For any $R>0$ and $\la \in \R$, $\left( r^{-\la},\left( \frac{\la+1}{2} \right)^{2}r^{-\lambda-2}+\frac{\Lambda^{2}}{R^{2}}r^{-\la}\right)$ is a $2$-Bessel pair on $\left(  0,R\right)  $ with a positive solution $\varphi=r^{\frac{\lambda+1}{2}}J_{0}\left(  \frac{\Lambda r}{R}\right)$.
  Moreover, Lamb's constant $\la_{0}$ is the largest value of $\Lambda$ for which $\varphi$ is increasing on $(0,R)$.
\end{example}

\begin{example}
  \label{B7}
  For any $R>0$ and $1<p<2$, the pair
  \[
    \left(  r, \frac{(p-1)^{2}}{p^{p}}r^{1-p}\left( \log\frac{R}{r} \right)^{-p}+\frac{(2-p)}{p^{p-1}}r^{1-p}\left( \log\frac{R}{r} \right)^{-p+1} \right)
  \]
  is a $p$-Bessel pair on $\left(0,R\right)$ with $\varphi=\left( \log\frac{R}{r} \right)^{\frac{1}{p}}$.
  Observe that at $p=2$, this pair reduces to the known $2$-Bessel pair used in \cite{LLZ20}.
\end{example}

\begin{proof}
  Consider only $r\in(0,R)$.
  Set $\varphi=\psi^{\frac{1}{p}}$, where
  $  \psi=\log\frac{R}{r}$.

  Set $V=r$, and set $W=cr^{1-p}\varphi^{-p^{2}}+Cr^{1-p}\varphi^{-p^{2}+p}$, where $c,C>0$ are to be determined.

  Compute
  \begin{align*}
    \varphi'&=\frac{1}{p}\psi'\psi^{\frac{1}{p}-1}=
    \frac{1}{p}\psi'\varphi^{1-p}\\
    \varphi''&=\frac{1}{p}\psi''\varphi^{1-p}-\frac{(p-1)}{p}\psi'\varphi'\varphi^{-p}\\
    &=\frac{1}{p}\psi''\varphi^{1-p}-\frac{(p-1)}{p^{2}}(\psi')^{2}\varphi^{1-2p}\\
    |\varphi'|^{p-2}\varphi'&=|\frac{1}{p}\psi'|^{p-2}\varphi^{(p-2)(1-p)}\frac{1}{p}\psi'\varphi^{1-p}\\
    &=|\frac{1}{p}\psi'|^{p-2}\frac{1}{p}\psi'\varphi^{-p^{2}+2p-1}\\
    |\varphi'|^{p-2}\varphi''&=|\frac{1}{p}\psi'|^{p-2}\varphi^{(p-2)(1-p)}\left[  \frac{1}{p}\psi''\varphi^{1-p}-\frac{(p-1)}{p^{2}}(\psi')^{2}\varphi^{1-2p} \right]\\
    &=\frac{1}{p^{p-1}}|\psi'|^{p-2}\psi''\varphi^{-p^{2}+2p-1}-(p-1)|\frac{1}{p}\psi'|^{p}\varphi^{-p^{2}+p-1}.
  \end{align*}
  Consequently,
  \begin{align*}
   & V'|\varphi'|^{p-2}\varphi'+(p-1)V|\varphi'|^{p-2}\varphi''+W\varphi^{p-1}\\= &V'|\frac{1}{p}\psi'|^{p-2}\frac{1}{p}\psi'\varphi^{-p^{2}+2p-1}
    +(p-1)V\frac{1}{p^{p-1}}|\psi'|^{p-2}\psi''\varphi^{-p^{2}+2p-1}\\
    &-(p-1)^{2}V|\frac{1}{p}\psi'|^{p}\varphi^{-p^{2}+p-1}
    +W\varphi^{p-1}.
  \end{align*}
  A quick consequence of this computation is that $V$, $W$, and $\varphi$ must satisfy
  \begin{align*}
    V'|\psi'|^{p-2}\psi'\psi+(p-1)V|\psi'|^{p-2}\psi''\psi-(p-1)^{2}pV|\psi'|^{p}+p^{p-1}W\psi^{p}=0
  \end{align*}
  if the solution $\varphi$ is to be of the form $\psi^{\frac{1}{p}}$.

  Therefore,
  \begin{align*}
   & V'|\varphi'|^{p-2}\varphi'+(p-1)V|\varphi'|^{p-2}\varphi''+W\varphi^{p-1}\\=&|\frac{1}{p}\psi'|^{p-2}\frac{1}{p}\psi'\varphi^{-p^{2}+2p-1}
    +(p-1)r\frac{1}{p^{p-1}}|\psi'|^{p-2}\psi''\varphi^{-p^{2}+2p-1}\\
    &-(p-1)^{2}r|\frac{1}{p}\psi'|^{p}\varphi^{-p^{2}+p-1}
    +cr^{1-p}\varphi^{-p^{2}+p-1}
    +Cr^{1-p}\varphi^{-p^{2}+2p-1}
  \end{align*}
  Observing that $\psi'=-\frac{1}{r}$ and $\psi''=\frac{1}{r^{2}}$, one finds this is equal to
  \begin{align*}
    &=-\frac{1}{p^{p-1}}r^{1-p}\varphi^{-p^{2}+2p-1}
    +(p-1)\frac{1}{p^{p-1}}r^{1-p}\varphi^{-p^{2}+2p-1}\\
    &-(p-1)^{2}\frac{1}{p^{p}}r^{1-p}\varphi^{-p^{2}+p-1}
    +cr^{1-p}\varphi^{-p^{2}+p-1}
    +Cr^{1-p}\varphi^{-p^{2}+2p-1}.
  \end{align*}
  Thus, one only needs to choose
  $$
    c=\frac{(p-1)^{2}}{p^{p}};\,\,
    C=\frac{(2-p)}{p^{p-1}}$$
  and observe that $2-p>0$ for $1<p<2$ to conclude that $(V,W)$ is a Bessel pair.

\end{proof}

\subsection{Geometric Preliminaries}%%%

In this section we state several geometric results pertaining to the distance function to the boundary of a domain.
%TODO: Is the following okay to say.
Most of the results here are recorded in the literature, with the possible exceptions Proposition \ref{prop:harmonic-implies-flat} and Corollary \ref{cor:bounded-implies-not-bounded}; we are unaware if these latter two results are recorded in the literature.

Fix a domain $\Omega\subsetneq\R^{N}$ with $C^{2}$ boundary and let $d_{\Omega}(x)=\dist(x,\p\Omega)$.
Let $k_{1}(y)$, $\ldots$, $k_{N-1}(y)$ be the principal curvatures of $\p\Omega$ at the point $y\in\p\Omega$.
The sign convention is chosen such that convex domains have non-positive principal curvatures.
Next, let $H(y)$ denote the mean curvature of $\p\Omega$ at $y\in\p\Omega$, normalized so that
\[
  H(y)=\frac{1}{N-1}\sum_{j=1}^{N-1}k_{j}(y).
\]
For the reader's convenience, we recall the following notions of convexity:
\begin{center}
  \begin{tabular}{ccc}
    convex & $\iff$ & $k_{j} \leq 0$\\
    strongly convex & $\iff$ & $k_{j} < 0$\\
    weakly mean convex & $\iff$ & $H \leq 0$\\
    mean convex & $\iff$ & $H < 0$
  \end{tabular},
\end{center}
where the inequalities hold pointwisely on $\p\Omega$ and for $j=1,\ldots,N-1$.
%Note that the class of weakly mean convex domains includes domains which are the complements of minimal surfaces.

\medskip

Now, given $x \in \Omega$, let $N(x) = \left\{ y \in \p\Omega : d_{\Omega}(x) = |x-y| \right\}$ denote the near set of $x$ whose elements are called near points, noting that $N(x)$ may contain more than one near point.
In case $N(x) = \left\{ y \right\}$, we write $N(x)$ for $y$.
The skeleton of $\Omega$ is the set $S(\Omega)=\left\{ x\in\Omega:\#N(x)>1 \right\}$, i.e., the set of points with more than one near point.
The skeleton of $\Omega$ is precisely where $d_{\Omega}$ fails to be differentiable in $\Omega$.
The closure (relative to $\Omega$) of the skeleton is called the cut locus of $\Omega$ and is denoted by $\Sigma(\Omega)$, and its complement $G(\Omega)=\Omega \setminus \Sigma(\Omega)$ will be called the ``good'' set, i.e., the largest open subset of $\Omega$ such that every point in $G(\Omega)$ has a unique near point (this set was defined by Li and Nirenberg in \cite{LN}).
%In case $\Omega$ is $C^{2,1}$, then $\Omega \setminus G(\Omega)$ has zero Lebesgue measure (c.f. \cite{IT,LN}).
Two results relating $\Delta{d_{\Omega}}$ to the curvatures $k_{j},j=1,\ldots,N$, and $H$ are recalled (c.f. \cite[Sections 2.4 \& 2.5]{BET}).

\begin{lemmaalph}
  \label{lem:laplace-curvature-identity}
  Let $\Omega \subset\R^{N}$, $N\geq2$, have $C^{2}$ boundary.
  Then $d_{\Omega} \in C^{2}(G(\Omega))$ and for $x \in G(\Omega)$ and $y=N(x)$ there holds
  \begin{align*}
    \Delta{d_{\Omega}(x)}&=\sum_{j=1}^{N-1}\frac{k_{j}(y)}{1+d_{\Omega}(x)k_{j}(y)}
  \end{align*}
  and
  \begin{align*}
    1&+d_{\Omega}(x)k_{j}(y)>0.
  \end{align*}
\end{lemmaalph}

\begin{lemmaalph}
  \label{lem:laplace-curvature-inequality}
  Let $\Omega \subset\R^{N}$, $N\geq2$, have $C^{2}$ boundary.
  Then for $x \in G(\Omega)$ and $y=N(x)$ there holds
  \begin{align*}
    \Delta{d_{\Omega}(x)}&\leq\frac{(N-1)H(y)}{1+d_{\Omega}(x)H(y)}\leq(N-1)H(y)\\
    \esssup_{\Omega}\Delta{d_{\Omega}}&=\esssup_{\p\Omega}(N-1)H.
  \end{align*}
\end{lemmaalph}

%In particular, the distance function $d_{\Omega}$ is smooth almost everywhere in $\Omega$.
%This is necessary for many of the following computations to hold without significant modification.

What is more is that outside of the skeleton the gradient of $d_{\Omega}$ takes a particularly nice form.
This is the content of the following theorem which was proved by Motzkin in \cite{motzkin1935quelques}.

\begin{theoremalph}
  \label{thm:gradient-of-distance-function}
  Let $\Omega \subset\R^{N}$, $N\geq2$, have $C^{2}$ boundary.
  If $x\in\Omega$ and $\#N(x)=1$, then $d_{\Omega}$ is differentiable at $x$ and
  \[
    \grad{d_{\Omega}}(x)=\frac{x-N(x)}{|x-N(x)|}.
  \]
  Moreover, $\grad{d_{\Omega}}$ is continuous where it is defined.
\end{theoremalph}

It is worth mentioning that the size of cut locus depends on the regularity of the boundary.
For example, in \cite[page 10]{MM}, they provide an example of a convex open $\Omega\subset\R^{2}$ with $C^{1,1}$ boundary and such that $\Sigma(\Omega)$ has non-zero Lebesgue measure.
On the other hand, if $\Omega$ is bounded and has a $C^{2,1}$ boundary, then $\Sigma$ is arcwise connected and has finite $(N-1)$-dimensional Hausdorff measure (and hence is a Lebesgue null set); see \cite{IT,LN}.

Next, we show that, if $\Omega$ is $C^{2}$ and $\Delta d_{\Omega} = 0$ near $\p\Omega$, then $k_{j}(N(x))=0$, $j=1,\ldots,N-1$, for $x$ close to $\p\Omega$.
Combining the following proposition with Motzkin's theorem (Theorem \ref{thm:motzkins-theorem}), we show that, if $\Delta d_{\Omega} = 0$ in the sense of distributions in $\Omega$ and $\Omega$ has $C^{2}$ boundary, then $\Omega$ is unbounded and has flat boundary.

\begin{proposition}
  \label{prop:harmonic-implies-flat}
  Let $\Omega \subset\R^{N}$, $N\geq2$, have $C^{2}$ boundary and let $y \in \p\Omega$.
  If $\Delta d_{\Omega} = 0$ in a small neighborhood of $y$, then $k_{j}(y) = 0$ for $j=1,\ldots,N-1$.
\end{proposition}

\begin{proof}
  Since $\Omega$ has $C^{2}$ boundary, we have that, for small $\e>0$, there holds $B(y,\e) \cap \Omega \subset G(\Omega)$.
  Let $x$ be in this intersection and such that $N(x) = y$.
  By Lemma \ref{lem:laplace-curvature-identity}, we have
  \begin{align*}
    \Delta{d_{\Omega}(x')}&=\sum_{j=1}^{N-1}\frac{k_{j}(y)}{1+d_{\Omega}(x')k_{j}(y)}
  \end{align*}
  for all $x'$ belonging to the line connecting $x$ to $y$ (since $N(x') = y$ for such $x'$), and $1+d_{\Omega}(x')k_{j}(y)>0$.
  For simplicity, let $k_{j}$ denote $k_{j}(y)$.
  It is then easy to see that $\Delta d_{\Omega}(x')=0$ is equivalent to
  \[
    k_{1} \cdots k_{N-1} d_{\Omega}(x')^{N-1} + LOT = 0,
  \]
  where $LOT$ consists of lower powers of $d_{\Omega}(x')$ with coefficients in terms of the $k_{j}$.
  As the $k_{j}$ are constant along the line joining $x$ to $y$, it follows that at least one of the $k_{j}$ is zero; i.e., up to relabeling, we may write
  \begin{align*}
    \Delta{d_{\Omega}(x')}&=\sum_{j=1}^{N-2}\frac{k_{j}}{1+d_{\Omega}(x')k_{j}}.
  \end{align*}
  Proceeding inductively, we conclude that $k_{j}=0$ for $j=1,\ldots,N-1$.
\end{proof}

Before stating the corollary, we need to recall a part of Motzkin's Theorem (attributed to the independent works of Bunt \cite{bunt1934bijdrage} and Motzkin \cite{motzkin1935quelques}).

\begin{theoremalph}
  \label{thm:motzkins-theorem}
  Let $ U \subset\R^{N}$ be a domain and let $\operatorname{Ext} U = \R^{N} \setminus \overline U$ be its exterior.
  Then $\overline U$ is convex if and only if $d_{\operatorname{Ext} U}$ is differentiable at every point in $\operatorname{Ext} U$.
\end{theoremalph}

\begin{corollary}
  \label{cor:bounded-implies-not-bounded}
  Suppose $\Omega$ has $C^{2}$ boundary.
  If $\Delta d_{\Omega} = 0$ in the sense of distributions in $\Omega$, then $\Omega$ is unbounded and its boundary is flat.
\end{corollary}

\begin{proof}
  If $\Delta d_{\Omega} = 0$, then, by the hypoellipticity of $\Delta$, we may conclude that $d_{\Omega}$ is smooth in $\Omega$.
  By Theorem \ref{thm:motzkins-theorem}, we conclude that the exterior of $\Omega$ is convex, which is only possible if $\Omega$ is unbounded.
  Using Proposition \ref{prop:harmonic-implies-flat}, we may conclude that $\p\Omega$ is everywhere flat.
\end{proof}

\medskip

We end this section by recalling the distributional Laplacian of $d_{\Omega}$ and compute it explicitly using geodesic normal coordinates.
The following is likely known and may be found in \cite{S} in the case of manifolds with boundaries.

To begin, let $\Omega$ be $C^{2,1}$ with finite inradius and, for each $x\in\Omega$, let $\g_{x}(t)$ denote the straight line emanating from $x$ into $\Omega$ in the inward normal direction.
Next, define the number $c(x)>0$ as the number satisfying: $\g_{x}(t)$ minimizes the distance to $\p\Omega$ if and only if $t\in[0,c(x)]$.
% \textbf{TODO: double check the folllowing}
% Observe that
% \[
%   \Sigma = \bigcup_{x\in\Omega} \left\{ \g_{x}(c(x)) \right\},
% \]
% and that, since $\Omega$ has finite inradius, there holds $c(x) \in (0,\oo)$ for all $x \in \p \Omega$.
Observe that, since $\Omega$ has finite inradius, there holds $c(x) \in (0,\oo)$ for all $x \in \p \Omega$.
The normal coordinates for a point $x \in G(\Omega)$ is then given by the tuple $(r,\sigma)$, where $0<r<c(\sigma)$ and $\sigma \in \p\Omega$.
Lastly, let $\theta(r,\sigma)$ denote the density of the Lebesgue measure written in normal coordinates: $dx = \theta(r,\sigma) dr d\mcH^{N-1}(\sigma)$.
In particular, if $f$ is integrable on $\Omega$, then
\[
  \int\limits_{\Omega} f dx = \int\limits_{\p\Omega} \int\limits_{0}^{c(\sigma)}f(r,\sigma) \theta(r,\sigma) dr d\mcH^{N-1}(\sigma).
\]
Moreover, $r \mapsto \theta(r,\sigma)$ is extended continuously to $[0,c(\sigma)]$ by setting
\[
  \theta(c(\sigma),\sigma) = \lim_{r \to c(\sigma)^{-}} \theta(r,\sigma).
\]

Now, since $d_{\Omega}$ is only Lipschitz continuous, it need not be twice weakly differentiable on $\Omega$, and therefore one needs to define $\Delta d_{\Omega}$ in the sense of distributions:
\[
  \gen{\Delta d_{\Omega}, \phi} = \int\limits_{\Omega} d_{\Omega}(x) \Delta \phi(x) dx = - \int\limits_{\Omega} \grad d_{\Omega}(x) \cdot \grad \phi(x) dx
\]
for $\phi \in C_{0}^{\oo}(\Omega)$, where the second equality follows from $d_{\Omega}$ being Lipschitz continuous.
If $G=G(\Omega)$ and $\Sigma = \Sigma(\Omega)$ denote the good set and cut locus, respectively, then $\Delta d_{\Omega}$ splits into two distributions:
\begin{equation}
  \Delta d_{\Omega} = \Delta_{G} d_{\Omega} + \Delta_{\Sigma} d_{\Omega},
  \label{eq:distributional-laplacian-decomposition}
\end{equation}
where $\Delta_{G} d_{\Omega}$ is the distribution defined by the function $\Delta (d_{\Omega} |_{G})$, and  $\Delta_{\Sigma} d_{\Omega}$ is defined by the pairing
\begin{equation}
  \gen{\Delta_{\Sigma} d_{\Omega}, \phi} = -\int\limits_{\p\Omega} \phi\left( c(x),x \right) \theta(c(x),x) d\mcH^{N-1}(x),
  \label{eq:distributional-laplacian-cut-locus}
\end{equation}
for $\phi \in C_{0}^{\oo}(\Omega)$ (note that the nonnegative Laplacian is used in \cite{S}, and whence the difference in sign).
In particular, since $\Omega$ is $C^{2,1}$, the $G$ has measure zero, and so $\Delta_{G}d_{\Omega}$ is defined almost everywhere in $\Omega$.
We thus use $\Delta d_{\Omega}$ to denote both the distributional Laplacian of $d_{\Omega}$, and the distribution $\Delta_{G} d_{\Omega}$, i.e., we will often write
\[
  \gen{\Delta_{G} d_{\Omega}, \phi} = \int\limits_{\Omega} \phi(x) \Delta d_{\Omega}(x) dx = \int\limits_{\Omega} d_{\Omega}(x) \Delta \phi(x) dx.
\]
Therefore, in case $\Omega$ is $C^{2,1}$ and if $F$ is regular enough function, there holds
\begin{equation}
  \int\limits_{\Omega} \grad F(x) \cdot \grad d_{\Omega}(x) dx = - \int\limits_{\Omega} F(x) \Delta d_{\Omega}(x) dx - \gen{ \Delta_{\Sigma} d_{\Omega}, F}.
  \label{eq:divergence-theorem-for-good-laplacian}
\end{equation}

\section{Main Results and Proofs}
\label{main-results}

\subsection{One-Dimensional Hardy Identities}%%%
\label{subsec3.1}
In this section, we prove one-dimensional Hardy identities with respect to $p$-Bessel pairs.
We recall the classical problem of determining necessary and sufficient conditions on a pair of weights $(V,W)$ for which there exists a positive finite constant $C$ such that there holds
\begin{equation}
  \int\limits_{0}^{\oo}\left| W(x) \int_{0}^{x}f(t) dt \right|^{p} dx \leq C  \int\limits_{0}^{\oo}|V(x)f(x)|^{p}dx,
  \label{eq:muckenhoupt-one-dim}
\end{equation}
where $1 \leq p \leq \oo$ and the differential Hardy inequality follows by taking $f(t) = F'(t)$ with $F(0) = 0$.
A classical necessary and sufficient condition for \eqref{eq:muckenhoupt-one-dim} to hold is for
\begin{equation}
  \sup_{r>0}\left[ \int\limits_{r}^{\oo} |W(x)|^{p} dx \right]^{\frac{1}{p}}\left[ \int\limits_{0}^{r}|V(x)|^{-p'}dx \right]^{\frac{1}{p'}}
  \label{eq:muckenhoupt-one-dim-condition}
\end{equation}
to be finite, where $\frac{1}{p} + \frac{1}{p'} = 1$; see for example \cite{Mu,T,To} and the references therein.
It is not hard to follow Muckenhoupt's proof in \cite{Mu} to show that an analogous statement holds on bounded intervals $(a,b)$ where $a$ and $b$ replace, respectively, the bounds $0$ and $\oo$ in \eqref{eq:muckenhoupt-one-dim} and \eqref{eq:muckenhoupt-one-dim-condition}.
The main results in this section make use of $p$-Bessel pairs to prove Hardy identities corresponding to the differential version of \eqref{eq:muckenhoupt-one-dim} on bounded intervals.

The proof of the main result in this subsection, Theorem \ref{thm:one-dimensional-hardy-identities}, will make it clear (without needing to consider boundary geometry) why we should expect distributional terms in Hardy identities for domains.
Moreover, Corollary \ref{cor:one-dimensional-hardy-identity-corollary-2} makes it clear why we should expect Lamb's constant to appear in the improved Hardy inequalities in \cite{AW}.
Lastly, the results from this section will be used for results pertaining to the mean distance function.%%%

\medskip

\begin{theorem}
  \label{thm:one-dimensional-hardy-identities}
  Let $a<b$ be two real numbers.
  Let $(V,W)$ be a $p$-bessel pair on the interval $(0,R')$ with $R'>R = \frac{b-a}{2}$, $1 < p<\oo$ and positive solution $\varphi$.
  Lastly, let $d(t)=\min\left\{ t-a,b-t \right\}$.
  Then, for $u\in{C_{0}^{\oo}(a,b)}$, there holds
  \begin{align*}
    \int\limits_{a}^{b}V(d(t))|u'(t)|^{p}dt-\int\limits_{a}^{b}W(d(t))|u(t)|^{p}dt&=\int\limits_{a}^{b}V(d(t))C_{p}\left(u'(t),\varphi(d(t))\left( \frac{u(t)}{\varphi(d(t))} \right)'\right)\\
    &+2\frac{V(R)|\varphi'(R)|^{p-2}\varphi'(R)}{\varphi(R)^{p-1}}\left|u\left(\frac{a+b}{2}\right)\right|^{p}.
  \end{align*}
  In particular, if $R$ is a critical point of $\varphi$, then
  \begin{align*}
    \int\limits_{a}^{b}V(d(t))|u'(t)|^{p}dt-\int\limits_{a}^{b}W(d(t))|u(t)|^{p}dt&=\int\limits_{a}^{b}V(d(t))C_{p}\left(u'(t),\varphi(d(t))\left( \frac{u(t)}{\varphi(d(t))} \right)'\right).
  \end{align*}
\end{theorem}%%%
\begin{remark}
  One may take the $p$-Bessel pair to be defined only on $(0,R)$ provided the limits in the following proof are well-defined.
  Moreover, an analogous result holds on an infinite interval where $R'=\oo$ and the constant term vanishes since the function is compactly supported.
\end{remark}

\begin{proof}
  For sake of notational convenience, we prove this theorem on the interval $(0,2)$ and let $d(t)=\min\left\{ t,2-t \right\}$ denote the distance to the boundary $\left\{ 0,2 \right\}$.
  Let $d'$ denote the almost everywhere defined derivative of $d$.
  Note that $R=1$.

  Let
  \[
    F(r)=V(r)|\varphi'(r)|^{p-2}\varphi'(r),
  \]
  and observe that, since $(V,W)$ is a $p$-Bessel pair, there holds
  \[
    -W(d(t))|\varphi(d(t))|^{p-2}\varphi(d(t))=F'(d(t))=F'(d(t))(d(t)')^{2}=(F(d(t)))'d'(t),
  \]
  %F is C^2 by Bessel pair equation and phi being positive.
  on $(0,1)\cup(1,2)$.
  Note that $d$ is only twice differentiable on $(0,1)\cup(1,2)$, and so, in the following, we will need to restrict integration to the intervals $(0,1-\e)$ and $(1+\e,2)$, and then take limits.
  Consequently, we find that, for $\e>0$ small, there holds
  \begin{align*}
    \int\limits_{0}^{1-\e}W(d(t))|u(t)|^{p}dt+\int\limits_{1+\e}^{2}W(d(t))|u(t)|^{p}dt&=-\int\limits_{0}^{1-\e}\frac{|u(t)|^{p}}{\varphi(d(t))^{p-1}}\left( F(d(t)) \right)'dt\\
    &+\int\limits_{1+\e}^{2}\frac{|u(t)|^{p}}{\varphi(d(t))^{p-1}}\left( F(d(t)) \right)'dt.
  \end{align*}
  Next, we compute
  \begin{align*}
    -\int\limits_{0}^{1-\e}\frac{|u(t)|^{p}}{\varphi(d(t))^{p-1}}\left( F(d(t)) \right)'dt&=\int\limits_{0}^{1-\e}F(d(t))\left( \frac{|u(t)|^{p}}{\varphi(d(t))^{p-1}} \right)'dt-\frac{|u(1-\e)|^{p}}{\varphi(1-\e)^{p-1}}F(1-\e)
  \end{align*}
  and
  \begin{align*}
    \int\limits_{1+\e}^{2}\frac{|u(t)|^{p}}{\varphi(d(t))^{p-1}}\left( F(d(t)) \right)'dt&=-\int\limits_{1+\e}^{2}F(d(t))\left( \frac{|u(t)|^{p}}{\varphi(d(t))^{p-1}} \right)'dt-\frac{|u(1+\e)|^{p}}{\varphi(1-\e)^{p-1}}F(1-\e),
  \end{align*}
  where we have used that $d(1+\e)=1-\e$.
  Taking $\e\to0$, we obtain
  \begin{align*}
    \int\limits_{0}^{2}W(d(t))|u(t)|^{p}dt&=\int\limits_{0}^{2}F(d(t))\left( \frac{|u(t)|^{p}}{\varphi(d(t))^{p-1}} \right)'d'(t)dt-2\frac{|u(1)|^{p}}{\varphi(1)^{p-1}}F(1).
  \end{align*}
  Next, computing
  \begin{align*}
    \left( \frac{|u(t)|^{p}}{\varphi(d(t))^{p-1}} \right)'&=\frac{p|u(t)|^{p-2}u(t)u'(t)}{\varphi(d(t))^{p-1}}-(p-1)\varphi'(d(t))d'(t)\frac{|u(t)|^{p}}{\varphi(d(t))^{p}},
  \end{align*}
  and observing
  \begin{align*}
    C_{p}\left(u'(t),\varphi(d(t))\left( \frac{u(t)}{\varphi(d(t))} \right)'\right)&=|u'(t)|^{p}+(p-1)\bigg\vert\frac{\varphi'(d(t))}{\varphi(d(t))}\bigg\vert^{p}|u(t)|^{p}\\
    &-p\bigg\vert\frac{\varphi'(d(t))}{\varphi(d(t))}\bigg\vert^{p-2}\frac{\varphi'(d(t))}{\varphi(d(t))}|u(t)|^{p-2}u(t)u'(t)d'(t),
  \end{align*}
  we find
  \begin{align*}
    F(d(t))\left( \frac{|u(t)|^{p}}{\varphi(d(t))^{p-1}} \right)'d'&=V(d(t))p\bigg\vert\frac{\varphi'(d(t))}{\varphi(d(t))}\bigg\vert^{p-2}\frac{\varphi'(d(t))}{\varphi(d(t))}|u(t)|^{p-2}u(t)u'(t)d'(t)\\
    &-V(d(t))(p-1)\bigg\vert\frac{\varphi'(d(t))}{\varphi(d(t))}\bigg\vert^{p}|u(t)|^{p}\\
    &=V(d(t))|u'(t)|^{p}-V(d(t))C_{p}\left(u'(t),\varphi(d(t))\left( \frac{u(t)}{\varphi(d(t))} \right)'\right).
  \end{align*}
  It follows that, at last, there holds
  \begin{align*}
    \int\limits_{0}^{2}V(d(t))|u'(t)|^{p}dt-\int\limits_{0}^{2}W(d(t))|u(t)|^{p}dt&=\int\limits_{0}^{2}V(d(t))C_{p}\left(u'(t),\varphi(d(t))\left( \frac{u(t)}{\varphi(d(t))} \right)'\right)\\
    &+2\frac{V(1)|\varphi'(1)|^{p-2}\varphi'(1)}{\varphi(1)^{p-1}}|u(1)|^{p}.
  \end{align*}
\end{proof}

We use Theorem \ref{thm:one-dimensional-hardy-identities} to obtain the following two corollaries.
For sake of clarity, we do not state the weighted versions of these corollaries (i.e., of the form given in Corollaries \ref{cor:weighted-hardy-identity-for-domains} and \ref{cor:avk-wirths-improvmenet}).

\begin{corollary}
  Assuming the hypotheses of Theorem \ref{thm:one-dimensional-hardy-identities}, there holds for $u \in C_{0}^{\oo}(a,b)$:
  \begin{align*}
    \int\limits_{a}^{b}|u'(t)|^{p}dt-\left( \frac{p-1}{p} \right)^{p}\int\limits_{a}^{b}\frac{|u(t)|^{p}}{d(t)^{p}}dt &=2\frac{p-1}{p}\left|u\left( \frac{b-a}{2} \right)\right|^{p}\\
    &+\int\limits_{a}^{b}C_{p}\left(u'(t),d(t)^{\frac{p-1}{p}}\left( \frac{u(t)}{d(t)^{\frac{p-1}{p}}} \right)'\right).
  \end{align*}
  \label{cor:one-dimensional-hardy-identity-corollary-1}
\end{corollary}

\begin{proof}
  This identity follows from applying Theorem \ref{thm:one-dimensional-hardy-identities} with the $p$-Bessel pair
  \[
    (V,W)=\left( 1, \left( \frac{p-1}{p} \right)^{p}r^{-p} \right)
  \]
  which has as a positive solution the function $\varphi=r^{\frac{p-1}{p}}$ on the infinite interval $(0,\oo)$, i.e., with $R'=\oo$.
\end{proof}

\begin{corollary}
  Assuming the hypotheses of Theorem \ref{thm:one-dimensional-hardy-identities}, there holds for $u \in C_{0}^{\oo}(a,b)$:
  \[
    \int\limits_{a}^{b}|u'(t)|^{2}dt - \frac{1}{4} \int\limits_{a}^{b} \frac{|u(t)|^{2}}{d(t)^{2}}dt = \frac{4\la_{0}^{2}}{(b-a)^{2}} \int\limits_{a}^{b}|u(t)|^{2}dt + \int\limits_{a}^{b} d(t) J_{0}\left( \frac{\la_{0}}{R}t \right)^{2} \left| \left( \frac{u(t)}{d(t)^{\frac{1}{2}} J_{0}\left( \frac{\la_{0}}{R}t \right)} \right)' \right|^{2}dt
  \]
  \label{cor:one-dimensional-hardy-identity-corollary-2}
\end{corollary}

\begin{proof}
  Let $R=\frac{b-a}{2}$ denote the inradius of $(a,b)$.
  This identity follows from applying Theorem \ref{thm:one-dimensional-hardy-identities} with the $p$-Bessel pair
  \[
    (V,W)=\left( 1,\left( \frac{1}{2} \right)^{2}r^{-2}+\frac{\la_{0}^{2}}{R^{2}} \right)
  \]
  which has as a positive solution the function $\varphi=r^{\frac{1}{2}}J_{0}\left( \frac{\la_{0}}{R}r \right)$ on the interval $(0,R')$, where $R'=\frac{\la_{1}}{\la_{0}}R>R$ (recall that $\la_{1}$ is the first positive zero of $J_{0}(t)$ and that $\la_{0}<\la_{1}$).
  See Example \ref{B2}.
  Note that $\varphi'(R)=0$.
\end{proof}

\subsection{Hardy Identities and Inequalities for Domains}\label{section:hardy-identities}%%%
In this section, we state and prove Hardy identities for general domains.
We emphasize the appearance of the distributional Laplacian of the distance function $d_{\Omega}$ due to $d_{\Omega}$ generally being at most Lipschitz continuous.
Subsequently, we will write the identities using geodesic normal coordinates when applicable and point out that, for weakly mean convex domains and appropriately chosen $p$-Bessel pairs, the identities are improvements of certain sharp Hardy inequalities.

We begin by stating the main identity theorem in this section.

\begin{theorem}%%%
  \label{thm:bessel-pair-hardy-identity-general-domain}
  Let $1 < p<\oo$, let $\Omega\subsetneq\R^{N}$ have inradius $0<\rho\leq\oo$, and let $0<R\leq\oo$ be such that $\rho < R$ when $\rho<\oo$ and $R=\oo$ when $\rho = \oo$.
  Suppose $(V,W)$ is a $p$-Bessel pair on $(0,R)$ with positive solution $\varphi$.
  Then, for $u\in{C_{0}^{\oo}(\Omega)}$, there holds
  \begin{equation}
    \begin{aligned}
      &\int\limits_{\Omega} V(d_{\Omega}(x)) |\grad u(x)|^{p} dx -\int\limits_{\Omega}W\left( d_{\Omega}(x) \right) |u(x)|^{p} dx  \\
      &= \int\limits_{\Omega}V(d_{\Omega}(x)) C_{p}\left( \grad u(x), \varphi(d_{\Omega}(x)) \grad\left( \frac{u(x)}{\varphi(d_{\Omega}(x))} \right) \right)\\
    %&+ \int\limits_{\Omega} V(d_{\Omega}(x)) |u(x)|^{p} \varphi(d_{\Omega}(x))^{1-p} |\varphi'(d_{\Omega}(x))|^{p-2} \varphi'(d_{\Omega}(x))\left[ \Delta d_{\Omega}(x) - \frac{\a-1}{d_{\Omega}(x)} \right] dx\\
      &- \gen{\Delta d_{\Omega},V \circ d_{\Omega}\left|\frac{\varphi' \circ d_{\Omega}}{\varphi \circ d_{\Omega}}\right|^{p-2}\frac{\varphi' \circ d_{\Omega}}{\varphi \circ d_{\Omega}} |u|^{p}}
    \end{aligned}
    \label{eq:main-hardy-identity-one}
  \end{equation}
  and
  \begin{equation}
    \begin{aligned}
      &\int\limits_{\Omega}{V(d_{\Omega}(x))}|\nabla{d_{\Omega}(x)}\cdot\nabla{u(x)}|^{p}dx -  \int\limits_{\Omega}W\left( d_{\Omega}(x) \right) |u(x)|^{p} dx \\
      &=\int\limits_{\Omega}{V(d_{\Omega}(x))}C_{p}\left(\nabla{d_{\Omega}(x)}\cdot{u(x)},\varphi(d_{\Omega}(x))\nabla{d_{\Omega}(x)}\cdot\nabla\left( \frac{u(x)}{\varphi(d_{\Omega}(x))} \right)\right)dx\\
    %&+ \int\limits_{\Omega} V(d_{\Omega}(x)) |u(x)|^{p} \varphi(d_{\Omega}(x))^{1-p} |\varphi'(d_{\Omega}(x))|^{p-2} \varphi'(d_{\Omega}(x))\left[ \Delta d_{\Omega}(x) - \frac{\a-1}{d_{\Omega}(x)} \right] dx.
      &- \gen{\Delta d_{\Omega},V \circ d_{\Omega}\left|\frac{\varphi' \circ d_{\Omega}}{\varphi \circ d_{\Omega}}\right|^{p-2}\frac{\varphi' \circ d_{\Omega}}{\varphi \circ d_{\Omega}} |u|^{p}}.
    \end{aligned}
    \label{eq:main-hardy-identity-two}
  \end{equation}
  Here, $\gen{\Delta d_{\Omega},\cdot}$ denotes the pairing with the distributional Laplacian of $d_{\Omega}$.
\end{theorem}

\def\zerodr{\int\limits_{\Omega}}
\begin{proof}
  Let
  \[
   % F(r)=r^{\a-1}V(r)|\varphi(r)'|^{p-2}\varphi'(r)
    F(r)=V(r)|\varphi(r)'|^{p-2}\varphi'(r)
  \]
  be given on $(0,R)$.
  Since $d_{\Omega}$ is Lipschitz and satisfies the eikonal equation, we have
  \[
    F'(d_{\Omega})=F'(d_{\Omega})\nabla{d_{\Omega}}\cdot\nabla{d_{\Omega}}=\nabla(F(d_{\Omega}))\cdot\nabla{d_{\Omega}}
  \]
  holds almost everywhere in $\Omega$.
  %Next, since $(r^{\a-1}V,r^{\a-1}W)$ form a $p$-Bessel pair, we have
  Next, since $(V,W)$ is a $p$-Bessel pair, we have
  \begin{equation}
    %-d_{\Omega}^{\a-1} W(d_{\Omega}) \varphi^{p-1}(d_{\Omega}) = F'(d_{\Omega})=\nabla(F(d_{\Omega}))\cdot\nabla{d_{\Omega}}
    - W(d_{\Omega}) \varphi^{p-1}(d_{\Omega}) = F'(d_{\Omega})=\nabla(F(d_{\Omega}))\cdot\nabla{d_{\Omega}}
    \label{eq:main-theorem-1-proof-equation-1}
  \end{equation}
  almost everywhere in $\Omega$.
  For convenience, let
  \[
    %\psi(x)=\frac{|u(x)|^{p}}{d_{\Omega}(x)^{\a-1}\varphi^{p-1}(d_{\Omega}(x))}.
    \psi(x)=\frac{|u(x)|^{p}}{\varphi^{p-1}(d_{\Omega}(x))}.
  \]
  %Thus, using the divergence theorem formula given in \eqref{eq:divergence-theorem-for-good-laplacian} and geodesic normal coordinates, we find
  Thus, using the divergence theorem formula given in \eqref{eq:divergence-theorem-for-good-laplacian}, we find
  \begin{align*}
    \int\limits_{\Omega}W\left( d_{\Omega}(x) \right) |u(x)|^{p} dx & = - \int\limits_{\Omega} \psi(x) \left[ -W(d_{\Omega}(x))\varphi^{p-1}(d_{\Omega}(x)) \right] dx\\
    &=-\int\limits_{\Omega} \psi(x) \grad F(d_{\Omega}(x)) \cdot \grad d_{\Omega}(x) dx\\
    &= \int\limits_{\Omega} F(d_{\Omega}(x)) \grad \psi(x) \cdot \grad d_{\Omega}(x) dx \\
    &+ \gen{\Delta d_{\Omega}, F(d_{\Omega}) \psi}.
    % &= \int\limits_{\Omega} F(d_{\Omega}(x)) \grad \psi(x) \cdot \grad d_{\Omega}(x) dx \\
    % &+ \int\limits_{\Omega} F(d_{\Omega}(x)) \psi(x) \Delta d_{\Omega}(x) dx \\
    % &- \int\limits_{\p\Omega}F(c(y)) \psi(c(y),y) \theta(c(y),y) d\mathcal{H}^{N-1}(y).
    %&+ \int\limits_{\p\Omega} V(c(y))\left|\frac{\varphi'(c(y))}{\varphi(c(y))}\right|^{p-2}\frac{\varphi'(c(y))}{\varphi(c(y))} |u(c(y),y)|^{p}   \theta(c(y),y) d\mcH^{N-1}(y).
  \end{align*}

  Next, we compute
  \begin{align*}
    F(d_{\Omega}(x))\psi(x) & =V(d_{\Omega}(x))\left|\frac{\varphi'(d_{\Omega}(x))}{\varphi(d_{\Omega}(x))}\right|^{p-2}\frac{\varphi'(d_{\Omega}(x))}{\varphi(d_{\Omega}(x))} |u(x)|^{p} \\
    F(d_{\Omega}(x)) \grad \psi(x) \cdot \grad d_{\Omega}(x) &= pV(d_{\Omega}(x)) \left| \frac{\varphi'(d_{\Omega}(x))}{\varphi(d_{\Omega}(x))} \right|^{p-2} \frac{\varphi'(d_{\Omega}(x))}{\varphi(d_{\Omega})(x)} |u(x)|^{p-2}u(x) \grad u(x) \cdot \grad d_{\Omega}(x) \\
    %& -\frac{\a-1}{d_{\Omega}}V(d_{\Omega}(x)) \left| \frac{\varphi'(d_{\Omega}(x))}{\varphi(d_{\Omega}(x))} \right|^{p-2} \frac{\varphi'(d_{\Omega}(x))}{\varphi(d_{\Omega}(x))} |u(x)|^{p} \\
    &-(p-1) V(d_{\Omega}(x)) \left| \frac{\varphi'(d_{\Omega}(x))}{\varphi(d_{\Omega}(x))} \right|^{p}|u(x)|^{p}
  \end{align*}
  and
  \begin{align*}
    C_{p}\left( \grad u(x), \varphi(d_{\Omega}(x)) \grad\left( \frac{u(x)}{\varphi(d_{\Omega}(x))} \right) \right) &= |\grad u(x)|^{p} + (p-1) \left| \frac{\varphi'(d_{\Omega}(x))}{\varphi(d_{\Omega}(x))} \right|^{p} |u(x)|^{p}\\
    &-p \left| \frac{\varphi'(d_{\Omega}(x))}{\varphi(d_{\Omega}(x))}\right|^{p-2} \frac{\varphi'(d_{\Omega}(x))}{\varphi(d_{\Omega}(x))}\\
    &\times |u(x)|^{p-2}u(x) \grad u(x) \cdot \grad d_{\Omega}(x).
  \end{align*}
  % Therefore, we obtain
  % \begin{align*}
  %   F(d_{\Omega}(x)) \div\left( \psi(x) \grad d_{\Omega}(x) \right) &= V(d_{\Omega}(x)) |\grad u(x)|^{p} - V(d_{\Omega}(x)) C_{p}\left( \grad u(x), \varphi(d_{\Omega}(x)) \grad\left( \frac{u(x)}{\varphi(d_{\Omega}(x))} \right) \right)\\
  %   %& -\frac{\a-1}{d_{\Omega}}V(d_{\Omega}(x)) \left| \frac{\varphi'(d_{\Omega}(x))}{\varphi(d_{\Omega}(x))} \right|^{p-2} \frac{\varphi'(d_{\Omega}(x))}{\varphi(d_{\Omega}(x))} |u(x)|^{p}.\\
  %   &+ V(d_{\Omega}(x)) |u(x)|^{p} \varphi(d_{\Omega}(x))^{1-p} |\varphi'(d_{\Omega}(x))|^{p-2} \varphi'(d_{\Omega}(x)) \Delta d_{\Omega}(x)
  % \end{align*}
  Putting this together, we have
  \begin{align*}
    &\int\limits_{\Omega} V(d_{\Omega}(x)) |\grad u(x)|^{p} dx -\int\limits_{\Omega}W\left( d_{\Omega}(x) \right) |u(x)|^{p} dx  \\
    &= \int\limits_{\Omega}V(d_{\Omega}(x)) C_{p}\left( \grad u(x), \varphi(d_{\Omega}(x)) \grad\left( \frac{u(x)}{\varphi(d_{\Omega}(x))} \right) \right)\\
    %&+ \int\limits_{\Omega} V(d_{\Omega}(x)) |u(x)|^{p} \varphi(d_{\Omega}(x))^{1-p} |\varphi'(d_{\Omega}(x))|^{p-2} \varphi'(d_{\Omega}(x))\left[ \Delta d_{\Omega}(x) - \frac{\a-1}{d_{\Omega}(x)} \right] dx\\
    &- \gen{\Delta d_{\Omega},V \circ d_{\Omega}\left|\frac{\varphi' \circ d_{\Omega}}{\varphi \circ d_{\Omega}}\right|^{p-2}\frac{\varphi' \circ d_{\Omega}}{\varphi \circ d_{\Omega}} |u|^{p}}.
  \end{align*}
  A similar computation provides
  \begin{align*}
    &\int\limits_{\Omega}{V(d_{\Omega}(x))}|\nabla{d_{\Omega}(x)}\cdot\nabla{u(x)}|^{p}dx -  \int\limits_{\Omega}W\left( d_{\Omega}(x) \right) |u(x)|^{p} dx \\
    &=\int\limits_{\Omega}{V(d_{\Omega}(x))}C_{p}\left(\nabla{d_{\Omega}(x)}\cdot{u(x)},\varphi(d_{\Omega}(x))\nabla{d_{\Omega}(x)}\cdot\nabla\left( \frac{u(x)}{\varphi(d_{\Omega}(x))} \right)\right)dx\\
    %&+ \int\limits_{\Omega} V(d_{\Omega}(x)) |u(x)|^{p} \varphi(d_{\Omega}(x))^{1-p} |\varphi'(d_{\Omega}(x))|^{p-2} \varphi'(d_{\Omega}(x))\left[ \Delta d_{\Omega}(x) - \frac{\a-1}{d_{\Omega}(x)} \right] dx.
    &- \gen{\Delta d_{\Omega},V \circ d_{\Omega}\left|\frac{\varphi' \circ d_{\Omega}}{\varphi \circ d_{\Omega}}\right|^{p-2}\frac{\varphi' \circ d_{\Omega}}{\varphi \circ d_{\Omega}} |u|^{p}}.
  \end{align*}
  These are the desired identities.

\end{proof}

\begin{remark}
  Similar to the one-dimensional case, it may possible to take $\rho = R$ when $\rho <\oo$ provided
  \[
    \gen{\Delta d_{\Omega},V \circ d_{\Omega}\left|\frac{\varphi' \circ d_{\Omega}}{\varphi \circ d_{\Omega}}\right|^{p-2}\frac{\varphi' \circ d_{\Omega}}{\varphi \circ d_{\Omega}} |u|^{p}}
  \]
  is well-defined.

\end{remark}

We first point out that, under a suitable condition on $d_{\Omega}$, the aforementioned identities in fact improve known weighted Hardy inequalities.
It should also be pointed out that these improvements are only meaningful when the given Hardy inequality is sharp and that choosing certain $p$-Bessel pairs may in fact provide sharp Hardy inequalities.
See the applications that follow the following corollary.

\begin{corollary}\label{cor:improvements-corollary}%%%
  Assuming the hypotheses of Theorem \ref{thm:bessel-pair-hardy-identity-general-domain} and that $-\varphi'\Delta d_{\Omega} \geq 0$ in the sense of distributions, the identities \eqref{eq:main-hardy-identity-one} and \eqref{eq:main-hardy-identity-two} are, respectively, improvements of the inequalities
  \begin{equation}
    \int\limits_{\Omega} V(d_{\Omega}(x)) |\grad u(x)|^{p} dx \geq \int\limits_{\Omega}W\left( d_{\Omega}(x) \right) |u(x)|^{p} dx
    \label{eq:improvements-corollary-equation1}
  \end{equation}
  and
  \begin{equation}
    \int\limits_{\Omega}{V(d_{\Omega}(x))}|\nabla{d_{\Omega}(x)}\cdot\nabla{u(x)}|^{p}dx \geq  \int\limits_{\Omega}W\left( d_{\Omega}(x) \right) |u(x)|^{p} dx,
    \label{eq:improvmenets-corollary-equation2}
  \end{equation}
  where $u \in C_{0}^{\oo}(\Omega)$.
  If $\Omega$ is additionally $C^{2}$, then there does not exist nontrivial function $u \in W_{0}^{1,p}(\Omega,V,W)$ for which \eqref{eq:improvements-corollary-equation1} or \eqref{eq:improvmenets-corollary-equation2} is an equality.
\end{corollary}

\begin{remark}
  We recall that $-\Delta d_{\Omega} \geq 0$ holds for convex domains (and is in fact equivalent to convexity when $N=2$) and for weakly mean convex domains in case $\Omega$ is $C^{2}$.
\end{remark}

\begin{proof}
  That the identities \eqref{eq:main-hardy-identity-one} and \eqref{eq:main-hardy-identity-two} are improvements follows from

  \[
    -\gen{\Delta d_{\Omega},V \circ d_{\Omega}\left|\frac{\varphi' \circ d_{\Omega}}{\varphi \circ d_{\Omega}}\right|^{p-2}\frac{\varphi' \circ d_{\Omega}}{\varphi \circ d_{\Omega}} |u|^{p}} \geq0,
  \]
  which follows from the nonnegativity of
  \[
    V \circ d_{\Omega}\left|\frac{\varphi' \circ d_{\Omega}}{\varphi \circ d_{\Omega}}\right|^{p-2}\frac{\varphi' \circ d_{\Omega}}{\varphi \circ d_{\Omega}} |u|^{p}
  \]
  and $-\Delta d_{\Omega}$.

  Concerning the statement about extremality, suppose $u \in W_{0}^{1,p}(\Omega,V,W)$ is a function not identically equal to zero.
  By Theorem \ref{thm:bessel-pair-hardy-identity-general-domain}, there holds
  \[
    \int\limits_{\Omega} V(d_{\Omega}(x)) |\grad u(x)|^{p} dx = \int\limits_{\Omega}W\left( d_{\Omega}(x) \right) |u(x)|^{p} dx
  \]
  if and only if the remainder term vanishes, i.e.,
  \begin{align*}
    &\int\limits_{\Omega}V(d_{\Omega}(x)) C_{p}\left( \grad u(x), \varphi(d_{\Omega}(x)) \grad\left( \frac{u(x)}{\varphi(d_{\Omega}(x))} \right) \right)\\
    &- \gen{\Delta d_{\Omega},V \circ d_{\Omega}\left|\frac{\varphi' \circ d_{\Omega}}{\varphi \circ d_{\Omega}}\right|^{p-2}\frac{\varphi' \circ d_{\Omega}}{\varphi \circ d_{\Omega}} |u|^{p}}=0.
  \end{align*}
  Under the assumptions on $\Delta d_{\Omega}$ and $\varphi'$, this holds if and only if (see Lemma \ref{lem:cp-properties})
  \[
    \varphi(d_{\Omega}(x)) \grad\left( \frac{u(x)}{\varphi(d_{\Omega}(x))} \right)=0
  \]
  and $\Delta d_{\Omega} = 0$ in the sense of distributions, i.e., if and only if $u(x) = c \varphi(d_{\Omega}(x))$ for some nonzero $c \in \R$ and $d_{\Omega}$ is harmonic.
  Now Corollary \ref{cor:bounded-implies-not-bounded} implies $\Omega$ is either a strip or halfspace.
  As such, $\varphi(d_{\Omega})$ is constant on hyperplanes contained in $\Omega$ which are parallel to $\p\Omega$ and so it is clear that $\varphi(d_{\Omega})$ fails to be integrable.
  The result follows.

  Similarly,
  \[
    \int\limits_{\Omega}{V(d_{\Omega}(x))}|\nabla{d_{\Omega}(x)}\cdot\nabla{u(x)}|^{p}dx =  \int\limits_{\Omega}W\left( d_{\Omega}(x) \right) |u(x)|^{p} dx
  \]
  for a nontrivial $u$ if and only if $u(x) = c \varphi(d_{\Omega}(x))$ for some nonzero $c \in \R$ and $d_{\Omega}$ is harmonic.
\end{proof}

Next, we point out that in case $\Omega$ is assumed to be $C^{2,1}$, then the remainder term defined in terms of the distribution $\Delta d_{\Omega}$ may be explicitly written in terms of curvature and geodesic normal coordinates (see \eqref{eq:distributional-laplacian-cut-locus} and the preceding discussion).
Due to its geometric interest, we demonstrate and record this in the following corollary.

\begin{corollary}\label{cor:bessel-pair-hardy-identity-weakly-mean-convex-domains}%%%
  In addition to the hypotheses of Theorem \ref{thm:bessel-pair-hardy-identity-general-domain} and assume $\Omega$ is $C^{2,1}$.
  Let $y=N(x)$ denote the near point of $x \in G$.
  Then, for $u \in C_{0}^{\oo}(\Omega)$, there holds (in geodesic normal coordinates)
  \begin{align*}
    &\int\limits_{\Omega} V(d_{\Omega}(x)) |\grad u(x)|^{p} dx -\int\limits_{\Omega}W\left( d_{\Omega}(x) \right) |u(x)|^{p} dx  \\
    &= \int\limits_{\Omega}V(d_{\Omega}(x)) C_{p}\left( \grad u(x), \varphi(d_{\Omega}(x)) \grad\left( \frac{u(x)}{\varphi(d_{\Omega}(x))} \right) \right)\\
    &- \int\limits_{\Omega} V(d_{\Omega}(x))\left|\frac{\varphi'(d_{\Omega}(x))}{\varphi(d_{\Omega}(x))}\right|^{p-2}\frac{\varphi'(d_{\Omega}(x))}{\varphi(d_{\Omega}(x))} |u(x)|^{p} \sum_{j=1}^{N-1}\frac{k_{j}(y)}{1+d_{\Omega}(x)k_{j}(y)} dx\\
    &+ \int\limits_{\p\Omega} V(c(y))\left|\frac{\varphi'(c(y))}{\varphi(c(y))}\right|^{p-2}\frac{\varphi'(c(y))}{\varphi(c(y))} |u(c(y),y)|^{p}   \theta(c(y),y) d\mcH^{N-1}(y),
  \end{align*}
  and
  \begin{align*}
    &\int\limits_{\Omega}{V(d_{\Omega}(x))}|\nabla{d_{\Omega}(x)}\cdot\nabla{u(x)}|^{p}dx - \int\limits_{\Omega}W\left( d_{\Omega}(x) \right) |u(x)|^{p} dx\\
    =&\int\limits_{\Omega}{V(d_{\Omega}(x))}C_{p}\left(\nabla{d_{\Omega}(x)}\cdot{u(x)},\varphi(d_{\Omega}(x))\nabla{d_{\Omega}(x)}\cdot\nabla\left( \frac{u(x)}{\varphi(d_{\Omega}(x))} \right)\right)dx\\
    %&+ \int\limits_{\Omega} V(d_{\Omega}(x)) |u(x)|^{p} \varphi(d_{\Omega}(x))^{1-p} |\varphi'(d_{\Omega}(x))|^{p-2} \varphi'(d_{\Omega}(x))\left[ \Delta d_{\Omega}(x) - \frac{\a-1}{d_{\Omega}(x)} \right] dx.
    &- \int\limits_{\Omega} V(d_{\Omega}(x)) |u(x)|^{p} \varphi(d_{\Omega}(x))^{1-p} |\varphi'(d_{\Omega}(x))|^{p-2} \varphi'(d_{\Omega}(x))  \sum_{j=1}^{N-1}\frac{k_{j}(y)}{1+d_{\Omega}(x)k_{j}(y)} dx\\
    &+ \int\limits_{\p\Omega} V(c(y))\left|\frac{\varphi'(c(y))}{\varphi(c(y))}\right|^{p-2}\frac{\varphi'(c(y))}{\varphi(c(y))} |u(c(y),y)|^{p}   \theta(c(y),y) d\mcH^{N-1}(y).
  \end{align*}
\end{corollary}

\begin{proof}
  First apply the decomposition \eqref{eq:distributional-laplacian-decomposition} to $\Delta d_{\Omega}$ to obtain
  \begin{align*}
    \gen{\Delta d_{\Omega},V \circ d_{\Omega}\left|\frac{\varphi' \circ d_{\Omega}}{\varphi \circ d_{\Omega}}\right|^{p-2}\frac{\varphi' \circ d_{\Omega}}{\varphi \circ d_{\Omega}} |u|^{p}} &= \gen{\Delta_{G} d_{\Omega},V \circ d_{\Omega}\left|\frac{\varphi' \circ d_{\Omega}}{\varphi \circ d_{\Omega}}\right|^{p-2}\frac{\varphi' \circ d_{\Omega}}{\varphi \circ d_{\Omega}} |u|^{p}} \\
    &+\gen{\Delta_{\Sigma} d_{\Omega},V \circ d_{\Omega}\left|\frac{\varphi' \circ d_{\Omega}}{\varphi \circ d_{\Omega}}\right|^{p-2}\frac{\varphi' \circ d_{\Omega}}{\varphi \circ d_{\Omega}} |u|^{p}}.
  \end{align*}
  Next, using that the cut locus $\Sigma$ has zero Lebesgue measure by $\Omega$ being $C^{2,1}$, we may use Lemma \ref{lem:laplace-curvature-identity} to obtain
  \begin{align*}
    &\gen{\Delta_{G} d_{\Omega},V \circ d_{\Omega}\left|\frac{\varphi' \circ d_{\Omega}}{\varphi \circ d_{\Omega}}\right|^{p-2}\frac{\varphi' \circ d_{\Omega}}{\varphi \circ d_{\Omega}} |u|^{p}} =\\
    &\int\limits_{\Omega} V(d_{\Omega}(x)) |u(x)|^{p} \varphi(d_{\Omega}(x))^{1-p} |\varphi'(d_{\Omega}(x))|^{p-2} \varphi'(d_{\Omega}(x))  \sum_{j=1}^{N-1}\frac{k_{j}(y)}{1+d_{\Omega}(x)k_{j}(y)} dx.
  \end{align*}
  Lastly, using the geodesic normal coordinate formula \eqref{eq:distributional-laplacian-cut-locus}, we have
  \begin{align*}
    &\gen{\Delta_{\Sigma} d_{\Omega},V \circ d_{\Omega}\left|\frac{\varphi' \circ d_{\Omega}}{\varphi \circ d_{\Omega}}\right|^{p-2}\frac{\varphi' \circ d_{\Omega}}{\varphi \circ d_{\Omega}} |u|^{p}} =\\
    &-\int\limits_{\p\Omega} V(c(y))\left|\frac{\varphi'(c(y))}{\varphi(c(y))}\right|^{p-2}\frac{\varphi'(c(y))}{\varphi(c(y))} |u(c(y),y)|^{p}   \theta(c(y),y) d\mcH^{N-1}(y).
  \end{align*}
  The desired identities follow.
\end{proof}

\begin{remark}
  %When $\rho<\oo$, it is possible to have $\rho = R$ provided $V(R)\left|\frac{\varphi'(R)}{\varphi(R)}\right|^{p-2}\frac{\varphi'(R)}{\varphi(R)}$ is well-defined.
  %Such care is needed since the set $\left\{ x \in \Omega: d_{\Omega}(x) = \rho \right\}$ is contained in the cut locus, i.e., the support of $\Delta_{\Sigma} d_{\Omega}$.
  %Indeed, let $x$ be such that $d_{\Omega}(x) = \rho$.
  %Let $\ell$ be the line passing through $x$ and touching the near point $y$ on $\p\Omega$.
  %In the direction pointing from $y$ to $x$, there is then a point on this line which is further from $\p\Omega$ than $x$.
  %Thus, $x$ cannot have a unique near point.
  It is possible to weaken the $C^{2,1}$-smoothness condition on $\Omega$ provided that we assume either the cut locus has zero Lebesgue measure or $u \in C_{0}^{\oo}(G)$ (in which case, $\Delta d_{\Omega}|_{C_{0}^{\oo}(G)} = \Delta_{G} d_{\Omega}$).
\end{remark}

We now state applications of Theorem \ref{thm:bessel-pair-hardy-identity-general-domain} by choosing specific $p$-Bessel pairs.
For sake of refraining from repetition, the following results will be stated without the analogous refined inequalities with $|\grad{u}|^{p}$ replaced by $|\grad{d_{\Omega}}\cdot\grad{u}|^{p}$.
We will also continue to use the distributional pairing notation (namely $\gen{\Delta d_{\Omega}, \cdot }$) with the understanding that, if $\Omega$ is $C^{2,1}$, then one could write $\Delta d_{\Omega}$ explicitly as done in Corollary \ref{cor:bessel-pair-hardy-identity-weakly-mean-convex-domains}.
The first identity pertains to the classical weighted Hardy inequalities whose weights are of the form $d_{\Omega}(x)^{\b}$ for some $\b\in\R$.

\begin{corollary}%%%
  \label{cor:weighted-hardy-identity-for-domains}
  Suppose $\Omega\subset\R^{N}$, $N\geq2$, is a domain satisfying the hypotheses of Theorem \ref{thm:bessel-pair-hardy-identity-general-domain}.
  Let $1 < p<\oo$ and $\la\in\R$ satisfy $p+\la-1\neq0$.
  Then, for $u\in{C_{0}^{\oo}(\Omega)}$, there holds
  \begin{align*}
    \int\limits_{\Omega}\frac{|\grad{u(x)}|^{p}}{d_{\Omega}(x)^{\la}}dx&-\bigg\vert\frac{p+\la-1}{p}\bigg\vert^{p}\int\limits_{\Omega}\frac{|u(x)|^{p}}{d_{\Omega}(x)^{\la+p}}dx\\
    &=\int\limits_{\Omega}\frac{1}{d_{\Omega}(x)^{\la}}C_{p}\left( \grad{u(x)},d_{\Omega}(x)^{\frac{p+\la}{p}}\grad\left( d_{\Omega}(x)^{-\frac{\la+p}{p}}u(x) \right) \right)dx\\
    &- \frac{p+\la-1}{p}\bigg\vert\frac{p+\la-1}{p}\bigg\vert^{p-2}\gen{\Delta d_{\Omega} , |u|^{p}d_{\Omega}^{1-p-\la}}.
  \end{align*}
  Moreover, when $- (p+\la-1)\Delta d_{\Omega} \geq0$, then this identity is an improvement of
  \[
    \int\limits_{\Omega}\frac{|\grad{u(x)}|^{p}}{d_{\Omega}(x)^{\la}}dx \geq \bigg\vert\frac{p+\la-1}{p}\bigg\vert^{p}\int\limits_{\Omega}\frac{|u(x)|^{p}}{d_{\Omega}(x)^{\la+p}}.
  \]
\end{corollary}

\begin{proof}[Proof of Corollary \ref{cor:weighted-hardy-identity-for-domains}]
  This identity follows from applying Theorem \ref{thm:bessel-pair-hardy-identity-general-domain} to the $p$-Bessel pair
  \[
    (V,W)=\left( r^{-\la},\bigg\vert\frac{p+\la-1}{p}\bigg\vert^{p}r^{-\la-p} \right)
  \]
  from Example \ref{B1}, which has as a positive solution the function $\varphi=r^{\frac{p+\la-1}{p}}$ on the infinite interval $(0,\oo)$, i.e., with $R=\oo$.
  The statement about improvements follows from Corollary \ref{cor:improvements-corollary}.
\end{proof}

\begin{remark}
  In case $\Omega$ is additionally weakly mean convex or convex and $C^{2}$ in a neighborhood of point on the boundary $\p \Omega$, and if $\la=0$, then these identities improve the sharp Hardy inequalities for weakly mean convex domains as discussed in the introduction.
\end{remark}

Next, we obtain Hardy identities in the spirit of Theorem \ref{thm:brezis-marcus-lewis-li-li}.
\begin{corollary}%%%
  \label{cor:avk-wirths-improvmenet}
  Suppose $\Omega\subsetneq\R^{N}$, $N\geq2$, has finite inradius $R$.
  Then, for $u\in{C_{0}^{\oo}(\Omega)}$, there holds
  \begin{equation}
    \label{eq:brezis-marcus-type-result}
    \begin{aligned}
      \int\limits_{\Omega}\frac{|\grad{u(x)}|^{2}}{d_{\Omega}(x)^{\la}}dx&-\bigg\vert\frac{1+\la}{2}\bigg\vert^{2}\int\limits_{\Omega}\frac{|u(x)|^{2}}{d_{\Omega}(x)^{\la+2}}dx-\frac{\la_{0}^{2}}{R^{2}}\int\limits_{\Omega}\frac{|u(x)|^{2}}{d_{\Omega}(x)^{\la}}dx\\
      &=\int\limits_{\Omega}d_{\Omega}(x)|J_{0}(\tfrac{\la_{0}}{R}d_{\Omega}(x))|^{2}\bigg\vert\grad\left( \frac{u(x)}{d_{\Omega}(x)^{\frac{1+\la}{2}}J_{0}(\frac{\la_{0}}{R}d_{\Omega}(x))} \right)\bigg\vert^{2}dx\\
      &-\frac{1+\la}{2}\gen{\Delta d_{\Omega}, \frac{|u|^{2}}{d_{\Omega}^{\la+1}}\left[ \frac{1+\la}{2}\frac{1}{d_{\Omega}}+\frac{\la_{0}}{R}\frac{J_{0}'\left( \frac{\la_{0}}{R}d_{\Omega} \right)}{J_{0}\left( \frac{\la_{0}}{R}d_{\Omega} \right)} \right]},
    \end{aligned}
  \end{equation}
  where $\la_{0}$ is the first positive zero of
  \[
    \left( r^{\frac{1+\la}{2}}J_{0}(r) \right)'=\frac{1+\la}{2}r^{\frac{-1+\la}{2}}J_{0}(r)+r^{\frac{1+\la}{2}}J_{0}'(r)=0.
  \]
  Moreover, when $-(1+\la)\Delta d_{\Omega}\geq0$, then this identity is an improvement of
  \[
    \int\limits_{\Omega}\frac{|\grad{u(x)}|^{2}}{d_{\Omega}(x)^{\la}}dx\geq\bigg\vert\frac{1+\la}{2}\bigg\vert^{2}\int\limits_{\Omega}\frac{|u(x)|^{2}}{d_{\Omega}(x)^{\la+2}}dx+\frac{\la_{0}^{2}}{R^{2}}\int\limits_{\Omega}\frac{|u(x)|^{2}}{d_{\Omega}(x)^{\la}}dx.
  \]
\end{corollary}

\begin{proof}
  %This identity follows from applying Theorem \ref{thm:bessel-pair-hardy-identity-general-domain} with $\a=1$ to the $p$-Bessel pair
  This identity follows from applying Theorem \ref{thm:bessel-pair-hardy-identity-general-domain} to the $p$-Bessel pair
  \[
    (V,W)=\left( r^{-\la},\left( \frac{1+\la}{2} \right)^{2}r^{-\la-2}+\frac{\la_{0}^{2}}{R^{2}}r^{-\la} \right)
  \]
  from Example \ref{B2}, which has as a positive solution the function $\varphi=r^{\frac{1+\la}{2}}J_{0}\left( \frac{\la_{0}}{R}r \right)$ on the interval $(0,R)$.
  Concerning the statement about improvements, we have by elementary calculus arguments that there holds
  \[
    \frac{\varphi'(r)}{\varphi(r)}= \frac{1+\la}{2}\frac{1}{r}+\frac{\la_{0}}{R}\frac{J_{0}'\left( \frac{\la_{0}}{R}r \right)}{J_{0}\left( \frac{\la_{0}}{R}r \right)}\geq0
  \]
  for $0\leq r \leq R$.
  Thus the last statement follows from Corollary \ref{cor:improvements-corollary}.
\end{proof}

\begin{remark}
  Note that, if $\la=0$, then Corollary \ref{cor:avk-wirths-improvmenet} improves Theorem \ref{thm:avk-wirths-theorem} and extends it to the case $\Omega$ is a weakly mean convex domain.
  We also mention that, if $c(y) = R$ for almost every $y \in \p\Omega$, then
  \[
    \frac{1+\la}{2}\int\limits_{\p\Omega}\frac{|u(x)|^{2}}{c(y)^{\la+1}} \left[ \frac{1+\la}{2}\frac{1}{c(y)}+\frac{\la_{0}}{R}\frac{J_{0}'\left( \frac{\la_{0}}{R}c(y) \right)}{J_{0}\left( \frac{\la_{0}}{R}c(y) \right)} \right] \theta(c(y),y) d\mathcal{H}^{N-1}(y) =0,
  \]
  and this follows by definition of $\la_{0}$.
  The geometry of the domain thus dictates the appearance of the boundary integral.
\end{remark}

\subsection{Mean Distance Function Results}
\label{section:mean-distance-function-results}

We now state and prove the main results concerning the mean distance function.
The first result concerns Hardy identities and resulting inequalities with respect to the spherical mean weights $\tilde{V}_{M,p}$, $V_{M,p}$ and $W_{M,p}$ (see \eqref{eq:spherical-means-of-bessel-pairs} for their definitions).
This result identifies a new class of weights for which a Hardy-type inequality holds.
%The second and main result concerns the lower spectral bounds for the eigenvalues of the Dirichlet problem \eqref{eq:dirichlet-problem}.
In order to set up these results we need to introduce two geometric concepts.
%TODO: Is this okay to say.
We are unaware if these concepts (namely, the linear essential diameter and $\nu$-skeleton defined below) have been recorded in the literature.

We begin by introducing what we will call the (lineal) essential diameter.
First fix a domain $\Omega \subsetneq \R^{N}$.
Thus, for each $\nu\in{S^{N-1}}$ (identified with the corresponding direction vector in $\R^{N}$), let $\p_{\nu}$ be the directional derivative in the direction $\nu$ and fix coordinates $x=(x',t)$, where $t = \nu \cdot x\in\R$ is the $\nu$-coordinate of $x$ and $(x',0)\in \nu^{\perp}$, the orthogonal complement of $\nu$.
We will identify the orthogonal complement $\nu^{\perp}$ with $\R^{N-1}$ by identifying the point $(x',0) \in \R^{N}$ with $x' \in \R^{N-1}$.
We may then identify each tuple $(x',\nu) \in \R^{N-1} \times S^{N-1}$ with the line parallel to $\nu$ and which passes through the point in $\R^{N}$ with coordinates $(x',0)$ as defined above.
(Of course $(x',\nu)$ and $(x',-\nu)$ are identified with the same line.)

Thus, let $\ell_{x',\nu}=\left\{ (x',0) + s\nu: s\in\R \right\}$ denote the line passing through $(x',0)$ in the direction $\nu$ and let $[x',\nu]$ denote the disjoint (possibly infinite) family of connected components of $\ell_{x',\nu} \cap \Omega$; i.e., $[x',\nu]$ is a countable family of line segments open in $\ell_{x',\nu}$.
Now, the proof of the first result relies on applicability of one-dimensional Hardy identities along the intersection of $\Omega$ with almost every line (here, almost everywhere is with respect to the product measure on $\R^{N-1}\times S^{N-1}$).
Suppose $(V,W)$ is a $p$-Bessel pair on the interval $(0,R)$ and suppose we wish to apply the one-dimensional Hardy identity of Theorem \ref{thm:one-dimensional-hardy-identities} with respect to $(V,W)$ along any element of $[x',\nu]$ for almost every $(x',\nu)\in\R^{N-1}\times S^{N-1}$ for which $[x',\nu]\neq\emptyset$.
Then, in order to do this, it is necessary that $2R$ is at least the length of any element of $[x',\nu]$ for all such $x'\in\R^{N-1}$ and $\nu \in S^{N-1}$.
If $\Omega$ has finite diameter, then one may simply take $2R$ to be larger than the diameter of $\Omega$; however, a larger $R$ may result in a sub-optimal inequality and so it is desirable to determine the smallest possible $R$.
We therefore introduce the following notion: let $D_{x',\nu} \in [0,\oo]$ denote the supremum of the lengths of the line segments in $[x',\nu]$ and define the (lineal) \textit{essential diameter} of $\Omega$ to be
\begin{equation}
  D_{\oo}(\Omega)=\esssup_{\nu\in{S^{N-1}}}\esssup_{x' \in \nu^{\perp}}D_{x',\nu};
  \label{eq:essential-diameter-definition}
\end{equation}
i.e., $D_{\oo}(\Omega)$ is the essential supremum of lengths of all line segments which are contained in $\Omega$.

Note that, unlike the inradius, $D_{\oo}(\Omega)$ is not sensitive to removing small sets from $\Omega$.
%This is an important distinction to make because, as is well known \textbf{TODO: cite papers which say the following},
%as was pointed out in \cite{D},
%the lowest eigenvalue of \eqref{eq:dirichlet-problem} in dimension $\geq3$ cannot be bounded from below by the inradius alone; e.g., one may decrease the inradius of a ball by removing line segments and not change the first eigenvalue since line segments have zero capacity.
Moreover $D_{\oo}(\Omega)$ is more concrete and tangible than the quasi-inradius \eqref{eq:quasi-inradius} of Davies.
We also mention that, as the proof of Theorem \ref{thm:mean-bessel-pair-hardy-identities} will demonstrate, the inradius may be too small to apply the $p$-Bessel pair equation \eqref{eq:p-bessel-pair-equation} along almost every line contained in $\Omega$.
This is apparently the case when $\Omega$ is, say, a non-equilateral rectangle in $\R^{2}$.
It is easy to see that, if $\Omega \subset \R^{N}$ is a domain with inradius $\d_{0}(\Omega)$ and essential diameter $D_{\oo}(\Omega)$, then
\[
  2\d_{0}(\Omega)\leq{D_{\oo}(\Omega)}\leq\diam(\Omega).
\]
%In fact, we have the following proposition.

%\begin{proposition}
%  If $\Omega\subset\R^{N}$ is a domain with finite inradius $\d_{0}(\Omega)$, then its essential diameter satisfies
%  \[
%    2\d_{0}(\Omega)\leq{D_{\oo}(\Omega)}\leq\diam(\Omega).
%  \]
%  Moreover
%  \[
%    2\d_{0}(\Omega)=D_{\oo}(\Omega)
%  \]
%  if and only if $\Omega$ is a ball of diameter $D_{\oo}(\Omega)$.
%\end{proposition}
%
%\begin{proof}
%  If $\Omega$ is a ball, then it is clear that $D_{\oo}(\Omega)=2\d_{0}(\Omega)$.
%  Suppose now that $\Omega$ is not a ball.
%  It is clear that there exists an $x\in\Omega$ such that $B(x,\d_{0})\subset\Omega$.
%  Since $\Omega$ is not a ball, the set $\Omega\setminus(\overline{B(x,\d_{0})} \cap \Omega)$ is open and thus has nontrivial intersection with a ball $B(y,\e)$ for some $\e>0$ and $y\in\p{B(x,\d_{0})}$.
%  This is enough to conclude that $D_{\oo}(\Omega)>2\d_{0}(\Omega)$.
%  That $D_{\oo}(\Omega) \leq \diam \Omega$ is obvious.
%\end{proof}

The following examples explicitly demonstrate that $D_{\oo}(\Omega)$ sits somewhere between twice the inradius and diameter of $\Omega$.

\begin{example}
  \noindent
  \begin{enumerate}
    \item For $\Omega=B(0,r)\setminus\left\{ 0 \right\}$, the inradius is $r$, and $D_{\oo}(\Omega)=2r$.
    \item Let $A = B(0,2) \setminus \overline{B(0,1)}$ be an annulus in $\R^{2}$.
      It is clear that $D_{\oo}(A) < \diam A$.
    \item The infinite strip $\left\{ (x,y) \in \R^{2} : -1<y<1 \right\}$ has finite inradius but $D_{\oo}(\Omega)=\diam \Omega = \oo$.
  \end{enumerate}
\end{example}

Now, given $(x',\nu) \in \R^{N-1} \times S^{N-1}$ and $\ell \in [x',\nu]$, let $m_{\ell}$ denote the midpoint of $\ell$ whenever $\ell$ is bounded; for $\ell$ unbounded we write $m_{\ell} = \oo$ and understand that $f(\oo) = 0$ for any compactly supported function $f:\R^{N} \to \R$.
We define the $\nu$-skeleton $\Sigma_{\nu}$ of a domain $\Omega$ to be the intersection
\[
  \Sigma_{\nu} = \overline{\left\{ m_{\ell} : \ell \in [x',\nu], x' \in \nu^{\perp} \right\}}\cap\Omega
\]
in $\Omega$ and note that $\Sigma_{\nu}$ may be empty for certain $\nu \in S^{N-1}$.
We mention a few examples.

\begin{example}\label{ex:nu-skeleton}
  \noindent
  \begin{enumerate}
    \item For a two-dimensional disk, $\Sigma_{\nu}$ is the diameter orthogonal to $\nu$.
    \item For a two-dimensional annulus, $\Sigma_{\nu}$ is the disjoint union of two lines and two circular arcs.
    \item If $\Omega$ is convex, then $\Sigma_{\nu}$ is the graph of a function defined over a subset of $\nu^{\perp}$.
    \item The infinite strip $\left\{ (x,y) \in \R^{2} : -1<y<1 \right\}$ is such that $\Sigma_{(1,0)} = \emptyset$ and $\Sigma_{(0,1)}$ is the line $y=0$.
  \end{enumerate}
\end{example}

% Next we show that $\Sigma_{\nu}$ contains all of its limits points which lie in $\Omega$, which implies in particular that, if $K\subset\Omega$ is compact, then so is $\Sigma_{\nu} \cap K$.
% Moreover, it implies that $\Sigma_{\nu}$ is always measurable.

% \begin{lemma}
%   If $x \in \Omega$ is a limit point of $\Sigma_{\nu}$, then $x \in \Sigma_{\nu}$.
% \end{lemma}
%
% \begin{proof}
%   As the result is local, we may assume without loss of generality that $\Omega$ is bounded and hence $\p\Omega$ is compact.
%   Since $x$ is a limit point, there is a sequence of lines $\ell_{j}$ whose endpoints lie on $\p\Omega$ and so that $m_{\ell_{j}}$ tends to $x$.
%   Let $a_{\ell_{j}},b_{\ell_{j}} \in \p\Omega$ denote the endpoints of $\ell_{j}$, supposing that the $\nu$-coordinate of $a_{\ell_{j}}$ is larger than that of $b_{\ell_{j}}$.
%   As $\p\Omega$ is compact, there exists a subsequence of $(a_{\ell_{j}},b_{\ell_{j}}) \in \p\Omega \times \p\Omega$ which converges to a point $(a,b) \in \p\Omega \times \p\Omega$ with the $\nu$ coordinate of $a$ larger than that of $b$.
%   Allowing the abuse of notation, let $(a_{\ell_{j}},b_{\ell_{j}})$ also denote the subsequence.
%   Clearly we have
%   \begin{align*}
%     d(x,a) &= \lim_{j} d(m_{\ell_{j}},a_{\ell_{j}}) = \frac{1}{2}\lim_{j}d(a_{\ell_{j}},b_{\ell_{j}}) = \frac{1}{2}d(a,b)\\
%     d(x,b) &= \frac{1}{2}d(a,b),
%   \end{align*}
%   and so $x \in \Sigma_{\nu}$.
% \end{proof}

Before we state and prove the first main result in this section, we need to establish the following technical lemma.

\begin{lemma}
  For each $\nu \in S^{N-1}$, there exists a Radon measure $\mu_{\nu}$ on $\Sigma_{\nu}$ such that
  \[
    \int\limits_{\nu^{\perp}} \sum_{\ell \in [x',\nu]} f(m_{\ell}) dx' = \int\limits_{\Sigma_{\nu}} f(\sigma) d\mu_{\nu}(\sigma),\quad f \in C_{0}(\Omega),
  \]
  where $C_{0}(\Omega)$ denotes the continuous functions with compact support in $\Omega$.
  % and such that $\mu_{\nu}$ is locally (i.e., on compact subsets) represented by a finite sum of pushforward measures supported on compact subsets of $\nu^{\perp}$.
  % If $\Omega$ is convex, then $\mu_{\nu} = \phi_{*}\mathcal{H}^{N-1} \lfloor A$, where $A$ is the projection of $\Sigma_{\nu}$ onto $\nu^{\perp}$ and $\phi:A\to\Sigma_{\nu}$ maps $x' \in A$ to the midpoint of $\ell_{x',\nu} \cap \Omega$.
  \label{lem:technical-lemma-one}
\end{lemma}

%\textbf{TODO: we should be able to more or less provide the measure constructively.}

\begin{proof}
  Let $f \in C_{0}(\Sigma_{\nu})$ and define $S_{f}(x') = \sum_{\ell \in [x',\nu]} f(m_{\ell})$.
  We first show $S_{f}$ is a measurable function on $\nu^{\perp}$; we achieve this by approximating $S_{f}$ by measurable functions.
  So, let $\tilde{f} \in C_{0}(\R^{n})$ be a Tietze extension of $f$.
  Given $z \in \R^{n}$ introduce the coordinates $(x',t)$ where $x'$ is the projection of $z$ to $\nu^{\perp}$ and $t = \gen{z,\nu}$; thus $\R^{n}$ is identified with the product space $\nu^{\perp} \times \R$.
  Let $\Sigma_{\nu}^{\e} = \Sigma_{\nu} + B(0,\e)$ be the $\e$-extension of $\Sigma_{\nu}$.
  Since $\tilde{f}\chi_{\Sigma_{\nu}^{\e}}$ is measurable on $\R^{n}$, we have by Fubini's theorem that $F_{\e}(x') := \frac{1}{2\e}\int_{\R} \tilde{f}(x',t)\chi_{\Sigma_{\nu}^{\e}}(x',t) dt$ defines a family of measurable functions on $\nu^{\perp}$.
  Let $K \subset \Sigma_{\nu}$ be the compact support of $f$.
  Fix $x' \in \nu^{\perp}$ and note that we may choose $\e>0$ small enough so that any two lines $\ell_{1},\ell_{2} \in [x',\nu]_{K}$ satisfy $d(m_{\ell_{1}},m_{\ell_{2}}) > 5 \e$.
  For $\ell \in [x',\nu]_{K}$, let $M_{\ell} = \gen{\ell,\nu}$.
  Thus each $m_{\ell}$ corresponds to the point $(x',M_{\ell}) \in \nu^{\perp} \times \R$.
  Therefore, for $\e$ small, there holds
  \[
    F_{\e}(x') = \frac{1}{2\e} \sum_{\ell \in [x',\nu]_{K}} \int_{M_{\ell} - \e }^{M_{\ell} + \e } \tilde{f}(x',t) dt.
  \]
  Lebesgue differentiation for continuous functions provides
  \[
    \lim_{\e \to 0} \frac{1}{2\e} \int_{M_{\ell} - \e}^{M_{\ell} + \e} \tilde{f}(x',t) dt = f(m_{\ell}).
  \]
  As a consequence, we get
  \[
    S_{f}(x') = \lim_{\e \to 0} F_{\e}(x'),
  \]
  showing that $S_{f}$ is a measurable function on $\nu^{\perp}$.

  Define the linear functional $\psi: C_{0}(\Sigma_{\nu}) \to \R$ by
  \[
    \psi(g) = \int_{\nu^{\perp}} \sum_{\ell \in [x',\nu]} g(m_{\ell})dx',\quad g \in C_{0}(\Sigma_{\nu}).
  \]
  Let $K \subset \Sigma_{\nu}$ be the compact support of $g$, noting that $K$ is compact in $\Omega$.
  Next let $[x',\nu]_{K}$ be those lines in $[x',\nu]$ which intersect $K$.
  We show that, as a function of $x' \in \R^{N-1}$, the cardinality function $N(x')=\# [x',\nu]_{K}$ is bounded.
  If this were not the case, then there would be a sequence of points $x_{k}' \in \R^{N-1}$ such that $N(x_{k}') \to \oo$.
  After relabeling the sequence $(x_{k})$, there then exists a sequence $(\ell_{k})$ of lines with $\ell_{k} \in [x_{k}',\nu]_{K}$ and such that $(m_{\ell_{k}})$ converges to a point $m_{*} \in K$ (by compactness) and such that the lengths of the $\ell_{k}$ tend to zero (by the boundedness of $K$).
  %Relabeling the sequence $\left\{ \ell_{k} \right\}$ by the lines corresponding to the convergent subsequence of midpoints, this implies that the endpoints of $\ell_{k}$ (which belong to $\p\Omega$) converge to $m_{*}$, contradicting $m_{*} \in \Omega$.
  This implies that the endpoints of $\ell_{k}$ (which belong to $\p\Omega$) converge to $m_{*}$, contradicting $m_{*} \in \Omega$.
  Therefore $|\psi(g)| \lesssim \norm{g}_{\oo}<\oo$ and so $\psi$ is well-defined.

  Now, observing that $\psi\geq0$, we conclude $\psi$ is a positive linear functional on $C_{0}(\Sigma_{\nu})$ and so we may apply the Riesz-Markov-Kakutani representation theorem to conclude the existence of a unique Radon measure $\mu_{\nu}$ on $\Sigma_{\nu}$ such that $\psi(g) = \int_{\Sigma_{\nu}}g(m) d\mu_{\nu}(m)$.
  Since each $f \in C_{0}(\Omega)$ defines a unique function $f|_{\Sigma_{\nu}} \in C_{0}(\Sigma_{\nu})$, we conclude
  \[
    \int\limits_{\nu^{\perp}} \sum_{\ell \in [x',\nu]} f(m_{\ell}) dx' = \psi(f|_{\Sigma_{\nu}}) = \int_{\Sigma_{\nu}} f|_{\Sigma_{\nu}}(m) d\mu_{\nu}(m),
  \]
  as desired.

\end{proof}

% \begin{remark}
%   We remark that $\mu_{\nu}$ has the interpretation as the $N-1$-dimensional Hausdorff measure of the projection of $\Sigma_{\nu}$ onto $\nu^{\perp}$ while counting multiplicities of the projection.
%   In particular, if $s$ is the Hausdorff dimension of $\Sigma_{\nu}$, then $\mu$ need not equal the $s$-dimensional Hausdorff measure on $\Sigma_{\nu}$.
%   For example, this is the case when $s>N-1$.
% \end{remark}

To find the Hardy identities for the mean distance functions we will need to introduce the following means, which we will call the spherical skeletal mean of a function $u \in C_{0}^{\oo}(\Omega)$:
\begin{equation}
  \mathcal{S}_{\Omega}[u] = \int\limits_{S^{N-1}} \int\limits_{\Sigma_{\nu}} u(m) d \mu_{\nu}(m) d\sigma(\nu).
  \label{eq:spherical-skeletal-mean}
\end{equation}

We are now prepared to state and prove the first main result in this section.

\begin{theorem}
  \label{thm:mean-bessel-pair-hardy-identities}
  Let $N\geq2$, let $1 < p < \oo$, let $\Omega\subsetneq\R^{N}$ be any domain, and let $(V,W)$ be a positive pair of $C^{1}$ functions.
  If $(V,W)$ is a $p$-Bessel pair on $(0,R)$ for some $R>\frac{1}{2}D_{\oo}(\Omega)$ or $R=\oo$ if $D_{\oo}(\Omega)=\oo$, then, for all $u\in{C_{0}^{\oo}(\Omega)}$, there holds
  \begin{align*}
    \int\limits_{\Omega}& \tilde{V}_{M,p}(x) |\grad u(x)|^{p}dx  = \int\limits_{\Omega} W_{M,p}(x) |u(x)|^{p} dx +2\mathcal{S}_{\Omega}\left[ V \circ \rho_{\nu} \left| \frac{\varphi' \circ \rho_{\nu}}{\varphi \circ \rho_{\nu}} \right|^{p-2}\frac{\varphi' \circ \rho_{\nu}}{\varphi \circ \rho_{\nu}} |u|^{p} \right]\\
    &+ \int\limits_{S^{N-1}}\int\limits_{\Omega} V(\rho_{\nu}(x)) C_{p}\left( \p_{\nu}u(x), \varphi(\rho_{\nu}(x)) \p_{\nu} \left( \frac{u(x)}{\varphi(\rho_{\nu}(x))} \right) \right) dx d\sigma(\nu).
  \end{align*}
  and, using $V_{M,p}\geq\tilde{V}_{M,p}$, there holds
  \begin{align*}
    \int\limits_{\Omega}& V_{M,p}(x) |\grad u(x)|^{p}dx  \geq \int\limits_{\Omega} W_{M,p}(x) |u(x)|^{p} dx +2\mathcal{S}_{\Omega}\left[ V \circ \rho_{\nu} \left| \frac{\varphi' \circ \rho_{\nu}}{\varphi \circ \rho_{\nu}} \right|^{p-2}\frac{\varphi' \circ \rho_{\nu}}{\varphi \circ \rho_{\nu}} |u|^{p} \right]\\
    &+ \int\limits_{S^{N-1}}\int\limits_{\Omega} V(\rho_{\nu}(x)) C_{p}\left( \p_{\nu}u(x), \varphi(\rho_{\nu}(x)) \p_{\nu} \left( \frac{u(x)}{\varphi(\rho_{\nu}(x))} \right) \right) dx d\sigma(\nu).
  \end{align*}
\end{theorem}

\begin{proof}
  This proof is inspired by the proof given in \cite[Theorem 3.3.2]{BET}.

  Given $a,b\in\R\cup\left\{ \pm\oo \right\}$ with $a<b$, let  $\rho(t)=\min\left\{ |t-a|,|t-b| \right\}$ be the distance function on $(a,b)$ to the boundary $\left\{ a,b \right\}$, and let $m_{ab}=\frac{a+b}{2}$ be the midpoint of $(a,b)$.
  If $a,b\in\R$, then suppose $\rho(m_{ab})= \frac{b-a}{2} <R$, and observe that $|\rho'|=1$ and $\frac{d^{2}}{dx^{2}}\rho=0$ hold almost everywhere.
  From Theorem \ref{thm:one-dimensional-hardy-identities}, we have that, for the $p$-Bessel pair $(V,W)$, the one-dimensional Hardy identity for $\rho$ is given by
  \begin{equation}
    \begin{aligned}
      \int\limits_{a}^{b}V(\rho(t))|u'(t)|^{p}dt-\int\limits_{a}^{b}W(\rho(t))|u(t)|^{p}dt&=\int\limits_{a}^{b}V(\rho(t))C_{p}\left(u'(t),\varphi(\rho(t))\left( \frac{u(t)}{\varphi(\rho(t))} \right)'\right)\\
      &+2\frac{V(\rho(m_{ab}))|\varphi'(\rho(m_{ab}))|^{p-2}\varphi'(\rho(m_{ab}))}{\varphi(\rho(m_{ab}))^{p-1}}|u(m_{ab})|^{p}.
    \end{aligned}
    \label{eq:hardy-identity-one-dimension}
  \end{equation}
%   \begin{equation}
%     \int\limits_{a}^{b}|u'(t)|^{p}dt=\left( \frac{p-1}{p} \right)^{p}\int\limits_{a}^{b}\frac{|u(t)|^{p}}{\rho(t)^{p}}dt+\int\limits_{a}^{b}V(\rho)C_{p}\left( u',\varphi\left( \frac{u}{\varphi} \right)' \right)dt,
%     \label{eq:hardy-identity-one-dimension}
%   \end{equation}
  where $u\in{C_{0}^{\oo}(a,b)}$.

  We note that $\ell_{x',\nu}$ is in general a (possibly infinite) disjoint union of line segments contained in $\Omega$.
  Moreover, using that $(V,W)$ is a $p$-Bessel pair on $(0,R)$ and $R>\frac{1}{2}D_{\oo}(\Omega)$, we have for almost every $x'\in\R^{N-1}$ such that $\ell_{x',\nu}\neq\emptyset$, there holds
  \[
    V'(r)|\varphi'(r)|^{p-1}+V(r)|\varphi'(r)|^{p-2}\varphi''(r)+W(r)\varphi(r)^{p-1}=0
  \]
  for $r=\rho_{\nu}(x',t)$ whenever $(x',t)\in\ell_{x',\nu}$.

  For each $\ell\in[x',\nu]$, let $m_{\ell}$ denote the midpoint of $\ell$.
  Using Fubini's theorem and \eqref{eq:hardy-identity-one-dimension}, we have that
  \begin{align*}
    \int\limits_{\Omega} V\left( \rho_{\nu}(x) \right) |\p_{\nu} u(x)|^{p}dx & = \int\limits_{\nu^{\perp}} \sum_{\ell \in [x',\nu]} \int\limits_{\ell} V(\rho_{\nu}(x',t)) |\p_{\nu} u(x',t)|^{p} dt dx'\\
    &= \int\limits_{\nu^{\perp}} \sum_{\ell \in [x',\nu]} \int\limits_{\ell} W(\rho_{\nu}(x',t)) |u(x',t)|^{p} dt dx' \\
    &+ \int\limits_{\nu^{\perp}} \sum_{\ell \in [x',\nu]} \int\limits_{\ell} V(\rho_{\nu}(x',t))\\
    &\times C_{p}\left( \p_{\nu}u(x',t), \varphi(\rho_{\nu}(x',t)) \p_{\nu} \left( \frac{u(x',t)}{\varphi(\rho_{\nu}(x',t))} \right) \right) dt dx'\\
    &+ 2\int\limits_{\nu^{\perp}} \sum_{\ell \in [x',\nu]} \frac{V(\rho_{\nu}(m_{\ell}))|\varphi'(\rho_{\nu}(m_{\ell}))|^{p-2}\varphi'(\rho_{\nu}(m_{\ell}))}{\varphi(\rho_{\nu}(m_{\ell}))^{p-1}}|u(m_{\ell})|^{p}dx'.
  \end{align*}
  By Lemma \ref{lem:technical-lemma-one}, we have
  \begin{align*}
    \int\limits_{\Omega} V\left( \rho_{\nu}(x) \right) |\p_{\nu} u(x)|^{p}dx &= \int\limits_{\Omega} W(\rho_{\nu}(x)) |u(x)|^{p} dx \\
    &+ \int\limits_{\Omega} V(\rho_{\nu}(x)) C_{p}\left( \p_{\nu}u(x), \varphi(\rho_{\nu}(x)) \p_{\nu} \left( \frac{u(x)}{\varphi(\rho_{\nu}(x))} \right) \right) dx\\
    &+ 2\int\limits_{\Sigma_{\nu}}\frac{V(\rho_{\nu}(m))|\varphi'(\rho_{\nu}(m))|^{p-2}\varphi'(\rho_{\nu}(m))}{\varphi(\rho_{\nu}(m))^{p-1}}|u(m)|^{p} d\mu_{\nu}(m).
  \end{align*}

  Now, integrating over $S^{N-1}$ with respect to $\nu$, and using $\p_{\nu} u = \nu \cdot \grad u = |\grad u | \cos (\nu , \grad u)$, where $(a,b)$ denotes the angle between two vectors $a,b\in\R^{N}$, we obtain
  \begin{align*}
    \int\limits_{S^{N-1}} &\int\limits_{\Omega} V(\rho_{\nu}(x)) | \cos(\nu , \grad u(x))|^{p}dx d\sigma(\nu)  = \int\limits_{\Omega} W_{M,p}(x) |u(x)|^{p} dx \\
    &+ \int\limits_{S^{N-1}}\int\limits_{\Omega} V(\rho_{\nu}(x)) C_{p}\left( \p_{\nu}u(x), \varphi(\rho_{\nu}(x)) \p_{\nu} \left( \frac{u(x)}{\varphi(\rho_{\nu}(x))} \right) \right) dx d\sigma(\nu)\\
    &+2 \mathcal{S}_{\Omega}\left[ \frac{V(\rho_{\nu})|\varphi'(\rho_{\nu})|^{p-2}\varphi'(\rho_{\nu})}{\varphi(\rho_{\nu})^{p-1}}|u|^{p} \right]
  \end{align*}
  Since the integration is being performed over the entire sphere $S^{N-1}$, we may write
  \[
    \int\limits_{S^{N-1}}V(\rho_{\nu}(x))|\cos\left( \nu,\grad{u} \right)|^{p}d\sigma(\nu)=\int\limits_{S^{N-1}}V(\rho_{\nu}(x))|\cos\left( \nu,\vec{e} \right)|^{p}d\sigma(\nu)=\tilde{V}_{M,p}(x),
  \]
  where $\vec{e}$ is any fixed unit vector in $\R^{N}$.

  At last, putting all this together, we obtain
  \begin{align*}
    \int\limits_{\Omega}& \tilde{V}_{M,p}(x) |\grad u(x)|^{p}dx  = \int\limits_{\Omega} W_{M,p}(x) |u(x)|^{p} dx +2 \mathcal{S}_{\Omega}\left[ \frac{V(\rho_{\nu})|\varphi'(\rho_{\nu})|^{p-2}\varphi'(\rho_{\nu})}{\varphi(\rho_{\nu})^{p-1}}|u|^{p} \right]\\
    &+ \int\limits_{S^{N-1}}\int\limits_{\Omega} V(\rho_{\nu}(x)) C_{p}\left( \p_{\nu}u(x), \varphi(\rho_{\nu}(x)) \p_{\nu} \left( \frac{u(x)}{\varphi(\rho_{\nu}(x))} \right) \right) dx d\sigma(\nu).
  \end{align*}
  Similarly, using $|\cos|\leq1$, and hence $\tilde{V}_{M,p} \leq V$, we obtain
  \begin{align*}
    \int\limits_{\Omega}& V_{M,p}(x) |\grad u(x)|^{p}dx  \geq \int\limits_{\Omega} W_{M,p}(x) |u(x)|^{p} dx +2 \mathcal{S}_{\Omega}\left[ \frac{V(\rho_{\nu})|\varphi'(\rho_{\nu})|^{p-2}\varphi'(\rho_{\nu})}{\varphi(\rho_{\nu})^{p-1}}|u|^{p} \right]\\
    &+ \int\limits_{S^{N-1}}\int\limits_{\Omega} V(\rho_{\nu}(x)) C_{p}\left( \p_{\nu}u(x), \varphi(\rho_{\nu}(x)) \p_{\nu} \left( \frac{u(x)}{\varphi(\rho_{\nu}(x))} \right) \right) dx d\sigma(\nu).
  \end{align*}
  This concludes the proof.

\end{proof}

We now give some applications of Theorem \ref{thm:mean-bessel-pair-hardy-identities}.
For sake of notational convenience, we shall write
\[
  \Xi(N,p)=\frac{\sqrt{\pi}\Gamma\left( \frac{N+p}{2} \right)}{\Gamma\left( \frac{p+1}{2} \right)\Gamma\left( \frac{N}{2} \right)},
\]
so that
\[
  d_{\Omega,M,p}(x)=\left( \Xi(N,p)\int\limits_{S^{N-1}}\frac{1}{\rho_{\nu}(x)^{p}}d\sigma(\nu) \right)^{-\frac{1}{p}}.
\]

First, just as the Bessel pair $(V,W)=\left(1,\left( \frac{p-1}{p} \right)^{p}r^{-p}\right)$ may be applied to obtain the classical sharp Hardy inequality \eqref{eq:hardy-inequality-general-domain}, the corresponding spherical means obtain the analogous Hardy inequality with respect to powers of $d_{\Omega,M,p}$.
This is the following corollary.

\begin{corollary}
  \label{cor:mean-distance-hardy-identity}
  For any domain $\Omega\subsetneq\R^{N}$, any $1<p<\oo$, and any $u\in{C_{0}^{\oo}(\Omega)}$, there holds
  \begin{align*}
    \frac{1}{\Xi(N,p)}\int\limits_{\Omega}|\grad{u(x)}|^{p}dx&=\left( \frac{p-1}{p} \right)^{p}\frac{1}{\Xi(N,p)}\int\limits_{\Omega}\frac{|u(x)|^{p}}{d_{\Omega,M,p}^{p}(x)}dx + 2\left( \frac{p-1}{p} \right)^{p-1} \mathcal{S}_{\Omega}\left[ |u|^{p}\rho_{\nu}^{1-p} \right]\\
    &+\int\limits_{\Omega}\int\limits_{S^{N-1}}C_{p}\left( \p_{\nu}u(x),\rho_{\nu}(x)^{\frac{p-1}{p}}\p_{\nu}\left( \frac{u(x)}{\rho_{\nu}(x)^{\frac{p-1}{p}}} \right) \right)d\sigma(\nu)dx,
  \end{align*}
  and, for $\la \in \R$, there holds
  \begin{align*}
    \frac{1}{\Xi(N,\la)}\int\limits_{\Omega}\frac{|\grad{u}|^{p}}{d_{\Omega,M,\la}(x)^{\la}}dx&\geq\bigg\vert\frac{p+\la-1}{p}\bigg\vert^{p}\frac{1}{\Xi(N,\la+p)}\int\limits_{\Omega}\frac{|u|^{p}}{d_{\Omega,M,\la+p}(x)^{\la+p}}dx\\
    &- \frac{p+\la-1}{p}\bigg\vert\frac{p+\la-1}{p}\bigg\vert^{p-2} \mathcal{S}_{\Omega}\left[ |u|^{p}\rho_{\nu}^{1-p-\la} \right]\\
    &+\int\limits_{\Omega}\int\limits_{S^{N-1}}\frac{1}{\rho_{\nu}(x)^{\la}}C_{p}\left( \p_{\nu}u(x),\rho_{\nu}(x)^{\frac{p+\la-1}{p}}\p_{\nu}\left( \frac{u(x)}{\rho_{\nu}(x)^{\frac{p+\la-1}{p}}} \right) \right)d\sigma(\nu)dx.
  \end{align*}
\end{corollary}

\begin{proof}
  This identity follows from applying Theorem \ref{thm:mean-bessel-pair-hardy-identities} to the $p$-Bessel pair
  \[
    (V,W)=\left( 1,\bigg\vert\frac{p-1}{p}\bigg\vert^{p}r^{-p} \right)
  \]
  which has as a positive solution the function $\varphi=r^{\frac{p-1}{p}}$ on the infinite interval $(0,\oo)$, i.e., with $R=\oo$.
  Indeed, we may compute
  \begin{align*}
    \tilde{V}_{M,p}(x)&=\int\limits_{S^{N-1}}|\cos(\nu\cdot\vec{e})|^{p}d\sigma(\nu)=\frac{1}{\Xi(N,p)}\\
    V_{M,p}(x)&=\int\limits_{S^{N-1}}d\sigma(\nu)=1\\
    W_{M,p}(x)&=\left( \frac{p-1}{p} \right)^{p}\int\limits_{S^{N-1}}\rho_{\nu}(x)^{-p}d\sigma(\nu)=\left( \frac{p-1}{p} \right)^{p}\frac{1}{\Xi(N,p)}\frac{1}{d_{\Omega,M,p}(x)^{p}}.
  \end{align*}
  Thus, using this and Lemma \ref{lem:varphi-stuff-computed}, the identity follows.

  Next, to obtain the inequality, we apply Theorem \ref{thm:mean-bessel-pair-hardy-identities} to the $p$-Bessel pair
  \[
    (V,W)=\left( r^{-\la},\bigg\vert\frac{p+\la-1}{p}\bigg\vert^{p}r^{-\la-p} \right)
  \]
  which has as a positive solution the function $\varphi=r^{\frac{p+\la-1}{p}}$ on the infinite interval $(0,\oo)$, i.e., with $R=\oo$.
  In this case, there is no general closed form $\tilde{V}_{M,p}$ (whence the inequality), but we do have
  \begin{align*}
    V_{M,p}(x)&=\int\limits_{S^{N-1}}\rho_{\nu}(x)^{-\la}d\sigma(\nu)=\frac{1}{\Xi(N,\la)}\frac{1}{d_{\Omega,M,\la}(x)^{\la}}\\
    W_{M,p}(x)&=\bigg\vert\frac{p+\la-1}{p}\bigg\vert^{p}\int\limits_{S^{N-1}}\rho_{\nu}(x)^{-\la-p}d\sigma(\nu)=\bigg\vert\frac{p+\la-1}{p}\bigg\vert^{p}\frac{1}{\Xi(N,\la+p)}\frac{1}{d_{\Omega,M,\la+p}(x)^{\la+p}}.
  \end{align*}
  Thus, using this and Lemma \ref{lem:varphi-stuff-computed}, the identity follows.

\end{proof}

Moreover,  we may use Theorem \ref{thm:mean-bessel-pair-hardy-identities} to obtain the following corollary that is in the spirit of Theorems \ref{thm:brezis-marcus} and \ref{thm:davies-tidblom-theorem}.

\begin{corollary}
  For any domain $\Omega\subset\R^{N}$, any $1<p<\oo$, and any $u\in{C_{0}^{\oo}(\Omega)}$, there holds
  \begin{align*}
    \int\limits_{\Omega}|\grad{u}|^{2}dx&=\frac{1}{4}\int\limits_{\Omega}\frac{|u(x)|^{2}}{d_{\Omega,M,2}^{2}(x)}dx+\frac{4N\la_{0}^{2}}{D_{\oo}(\Omega)^{2}}\int\limits_{\Omega}|u(x)|^{2}dx\\
    &+N\int\limits_{\Omega} \int\limits_{S^{N-1}} d_{\Omega}(x)|J_{0}(\tfrac{\la_{0}}{R}d_{\Omega}(x))|^{2}\bigg\vert\p_{\nu}\left( \frac{u(x)}{d_{\Omega}(x)^{\frac{1}{2}}J_{0}(\frac{\la_{0}}{R}d_{\Omega}(x))} \right)\bigg\vert^{2} d\sigma(\nu) dx\\
    &+ \mathcal{S}_{\Omega} \left[ \frac{|u|^{2}}{\rho_{\nu}} \left[ \frac{1}{2}\frac{1}{\rho_{\nu}}+\frac{\la_{0}}{R}\frac{J_{0}'\left( \frac{\la_{0}}{R}\rho_{\nu} \right)}{J_{0}\left( \frac{\la_{0}}{R}\rho_{\nu} \right)} \right]   \right]
  \end{align*}
  \label{cor:brezis-marcus-type-inequality-for-mean-distance}
\end{corollary}

\begin{proof}

  This identity follows from applying Theorem \ref{thm:mean-bessel-pair-hardy-identities} with $\a=1$ to the $p$-Bessel pair
  \[
    (V,W)=\left(1,\frac{1}{4}r^{-2}+\frac{4\la_{0}^{2}}{D_{\oo}(\Omega)^{2}}\right),
  \]
  which has as a positive solution the function $\varphi=r^{\frac{1}{2}}J_{0}\left( \frac{\la_{0}}{R}r \right)$ on the interval $(0,D_{\oo}(\Omega)/2)$ and is such that $\varphi'$ is also positive on this interval.
  Indeed, we may compute
  \begin{align*}
    \tilde{V}_{M,2}(x)&=\frac{1}{N}\\
    V_{M,2}(x)&=1\\
    W_{M,2}(x)&=\frac{1}{4}\int\limits_{S^{N-1}}\rho_{\nu}(x)^{-2}d\sigma(\nu)+\frac{4 \la_{0}^{2}}{D_{\oo}(\Omega)^{2}}\int\limits_{S^{N-1}}d\sigma(\nu)\\
    &=\frac{1}{4}\frac{1}{N}\frac{1}{d_{\Omega,M,2}(x)^{2}}+\frac{4 \la_{0}^{2}}{D_{\oo}(\Omega)^{2}},
  \end{align*}
  and the identity follows.
\end{proof}

Lastly, recall that, in \cite{D}, Davies used \eqref{eq:hardy-inequality-general-domain} with $p=2$ and $d=d_{\Omega,M,2}$ to establish the spectral lower bound
\[
  H_{0}\geq\frac{N}{4\mu^{2}},
\]
where $H_{0}$ is minus the Dirichlet Laplacian on $\Omega$ and
\[
  \mu = \sqrt{N} \sup\left\{ d_{\Omega,M,2}(x): x \in \Omega \right\}
\]
is the quasi-inradius of $\Omega$.
Such a result of course also holds for the $p$-Laplacian.
In a similar spirit, we may use Corollary \ref{cor:brezis-marcus-type-inequality-for-mean-distance} to obtain an improved spectral lower bound, and this is the content of the following result.

\begin{theorem}
  Let $H_{0}$ denote minus the Dirichlet Laplacian on $\Omega$, and let $\mu$ denote the quasi-inradius of $\Omega$.
  Then there holds
  \[
    H_{0}\geq\frac{N}{4\mu^{2}}+\frac{4N\la_{0}^{2}}{D_{\oo}(\Omega)^{2}}
  \]
  in the sense of bilinear forms.
  \label{cor:spectral-lower-bound-brezis-marcus-type}
\end{theorem}

\begin{proof}
  This follows directly from Corollary \ref{cor:brezis-marcus-type-inequality-for-mean-distance} since $d_{\Omega,M,2}\leq\mu$, and since the remainder term satisfies
  \begin{align*}
    &N\int\limits_{\Omega} \int\limits_{S^{N-1}} d_{\Omega}(x)|J_{0}(\tfrac{\la_{0}}{R}d_{\Omega}(x))|^{2}\bigg\vert\p_{\nu}\left( \frac{u(x)}{d_{\Omega}(x)^{\frac{1}{2}}J_{0}(\frac{\la_{0}}{R}d_{\Omega}(x))} \right)\bigg\vert^{2} d\sigma(\nu) dx\\
    &+ \mathcal{S}_{\Omega} \left[ \frac{|u|^{2}}{\rho_{\nu}} \left[ \frac{1}{2}\frac{1}{\rho_{\nu}}+\frac{\la_{0}}{R}\frac{J_{0}'\left( \frac{\la_{0}}{R}\rho_{\nu} \right)}{J_{0}\left( \frac{\la_{0}}{R}\rho_{\nu} \right)} \right]   \right]\geq0
  \end{align*}
  by account of
  \[
    \frac{1}{2}\frac{1}{r}+\frac{\la_{0}}{R}\frac{J_{0}'\left( \frac{\la_{0}}{R}r \right)}{J_{0}\left( \frac{\la_{0}}{R}r \right)} \geq0
  \]
  for $0\leq r \leq R$.
\end{proof}

\subsection{Sobolev Inequalities on Non-compact Non-complete Manifolds}\label{section:sobolev-inequality}%%%

In this last section, we use the $L^{2}$-Hardy identities in conjunction with the $L^{2}$-Hardy-Sobolev-Maz'ya inequality \eqref{eq:hsm} to establish Sobolev inequalities on certain manifolds with certain rough metrics.

To begin, let $\varphi = d_{\Omega}^{\frac{1}{2}}$ and note that $\varphi$ is positive in $\Omega$.
Next, let $g_{eu}$ denote the Euclidean metric on $\Omega$, and letting $A = \varphi^{\frac{2}{N-2}}$, define on $\Omega$ the metric $g = A^{2}g_{eu}$.
If $U$ is further assumed to be convex, then $g$ is a smooth metric by Theorem \ref{thm:motzkins-theorem}.
Since the Riemannian gradient $\grad_{g}$ and volume form $dV_{g}$ do not require smoothness of $g$, these are well-defined even if $g$ is not smooth.
Next, let $|\cdot|_{g}$ denote the norm with respect to $g$, and let $2^{*}=\frac{2N}{N-2}$ denote the Sobolev exponent.

\begin{theorem}
  \label{thm:sobolev-inequality-theorem}
  Suppose $\Omega$ supports the Hardy-Sobolev-Maz'ya inequality \eqref{eq:hsm} and $\Delta d_{\Omega} \geq 0$ in the sense of distributions.
  If $\Omega$ is endowed with the (generally rough) metric $g = d_{\Omega}^{\frac{2}{N-2}}|dx|^{2}$, then the following Sobolev inequality holds on $(\Omega,g)$
  \begin{equation}
    \int\limits_{\Omega}|\grad_{g} v |_{g}^{2} dV_{g} \geq C \left( \int\limits_{\Omega} |v|^{2^{*}} dV_{g} \right)^{\frac{2}{2^{*}}}
    \label{eq:sobolev-inequality}
  \end{equation}
  %for all functions $v$ such that $d_{\Omega}^{\frac{1}{2}}v \in C_{0}^{1,1}(\Omega)$, and where $C$ is a positive constant depending only on $\Omega$ and $N$.
  for all functions $v \in C_{0}^{\oo}(\Omega)$ and where $C$ is a positive constant depending independent of $v$.
  %such that $d_{\Omega}^{\frac{1}{2}}v \in C_{0}^{1,1}(\Omega)$.
\end{theorem}

\begin{proof}[Proof of Theorem \ref{thm:sobolev-inequality-theorem} (and Theorem \ref{thm:main-theorem3})]
  Let $u = d_{\Omega}^{\frac{1}{2}}v$.
  Under the conformal change $g_{eu} \to g = A^{2} g_{eu}$, we have the following identities
  \begin{align*}
    \grad_{g} &= A^{-2} \grad\\
    dV_{g} & =A^{N}dx\\
    |\grad_{g} u|_{g}^{2} &= A^{-2} |\grad u|^{2},
  \end{align*}
  where the Riemannian data $\grad$, $dx$, and $|\cdot|$ are all with respect to $g_{eu}$.
  It is then easy to see that
  \[
    \int\limits_{\Omega} |u|^{2*} dx = \int\limits_{\Omega}\left| \frac{u(x)}{d_{\Omega}(x)^{\frac{1}{2}}} \right|^{2^{*}} dV_{g}.
  \]
  Next, since it is assumed that \eqref{eq:hsm} is supported, there holds
  %Next, by Theorem \ref{thm:gkikas-hsm-inequality}, there holds
  \[
    \int\limits_{\Omega}|\grad{u}|^{2}dx-\frac{1}{2}\int\limits_{\Omega}\frac{|u|^{2}}{d_{\Omega}(x)^{2}}dx \geq C \left( \int\limits_{\Omega}|u|^{2^{*}} dx \right)^{\frac{2}{2^{*}}} = \left( \int\limits_{\Omega}\left| \frac{u(x)}{d_{\Omega}(x)^{\frac{1}{2}}} \right|^{2^{*}} dV_{g}\right)^{\frac{2}{2^{*}}}.
  \]

  On the other hand, Theorem \ref{thm:bessel-pair-hardy-identity-general-domain} yields
  \begin{align*}
    \int\limits_{\Omega}  |\grad u(x)|^{2} dx -\int\limits_{\Omega} \frac{|u(x)|^{2}}{d_{\Omega}(x)^{2}} dx &= \int\limits_{\Omega} d_{\Omega}(x) \left| \grad \left( \frac{u(x)}{d_{\Omega}(x)^{\frac{1}{2}}} \right) \right|^{2} dx- \gen{\Delta d_{\Omega},V\left|\frac{\varphi'}{\varphi}\right|^{p-2}\frac{\varphi'}{\varphi} |u|^{p}}\\
    &\leq  \int\limits_{\Omega} d_{\Omega}(x) \left| \grad \left( \frac{u(x)}{d_{\Omega}(x)^{\frac{1}{2}}} \right) \right|^{2} dx,
  \end{align*}
  where $\Delta d_{\Omega} \geq 0$ was used to conclude
  \[
    -\gen{\Delta d_{\Omega},V\left|\frac{\varphi'}{\varphi}\right|^{p-2}\frac{\varphi'}{\varphi} |u|^{p}} \leq 0.
  \]
  Putting this together, we get
  \begin{align*}
    \int\limits_{\Omega} |\grad_{g} v|^{2} dV_{g} = \int\limits_{\Omega} d_{\Omega}(x) \left| \grad v \right|^{2} dx \geq C \left( \int\limits_{\Omega}\left| v \right|^{2^{*}} dV_{g}\right)^{\frac{2}{2^{*}}},
  \end{align*}
  which establishes \eqref{eq:sobolev-inequality}.
  %where $v= d_{\Omega}^{-\frac{1}{2}}u$ is such that $d_{\Omega}^{\frac{1}{2}} v = u \in C_{0}^{1,1}(\Omega)$, which establishes \eqref{eq:sobolev-inequality}.
\end{proof}

Using Theorem \ref{thm:gkikas-hsm-inequality} for an exterior domains exterior to a weakly mean convex domain, one immediately obtains the following corollary.

\begin{corollary}
  Suppose $\Omega$ is an exterior domain satisfying \eqref{eq:super-harm-for-unbounded-domains} and which is the exterior of a weakly mean convex domain.
  If $\Omega$ is endowed with the metric $g = d_{\Omega}^{\frac{2}{N-2}}|dx|^{2}$, then the following Sobolev inequality holds on $(\Omega,g)$
  \begin{equation*}
    \int\limits_{\Omega}|\grad_{g} v |_{g}^{2} dV_{g} \geq C \left( \int\limits_{\Omega} |v|^{2^{*}} dV_{g} \right)^{\frac{2}{2^{*}}}.
  \end{equation*}
\end{corollary}

\bibliographystyle{abbrv}

\end{document}